\documentclass[letterpaper,10pt,reqno]{amsart}

\usepackage{graphicx} 
\usepackage{start4}
\usepackage{xcolor}
\usepackage{amsthm,amsmath,amssymb}
\usepackage[top=2cm,bottom=2cm,left=2.2cm,right=2.2cm]{geometry}
\usepackage[utf8]{inputenc}
\usepackage[colorlinks=true, allcolors=black]{hyperref}
\usepackage[capitalise,
            noabbrev,
            ]{cleveref}
\usepackage{tikz}
\usepackage{caption, subcaption}
\usepackage{mathtools}
\usepackage{enumitem}
\usepackage{mathrsfs}
\usepackage{titlesec}
\usepackage{etoolbox}
\usepackage{algorithm2e}
\usepackage{academicons}
\usepackage{url}

\titleformat{\section}{\centering\bfseries}{\thesection}{1em}{\MakeUppercase}
\titleformat{\subsection}{\bfseries}{\thesubsection}{1em}{}

\makeatletter
\patchcmd{\@thm}{\thm@headpunct{.}}{\thm@headpunct{}}{}{}
\makeatother

\newtheoremstyle{plain-nobrackets}
  {}   
  {}   
  {\itshape}  
  {}       
  {\bfseries} 
  {.}         
  {5pt plus 1pt minus 1pt} 
  {\thmname{#1}\thmnumber{ #2}\normalfont\thmnote{ #3}} 

\theoremstyle{plain-nobrackets}

\newtheorem{thm}{Theorem}[section]
\newtheorem{conj}[thm]{Conjecture}
\newtheorem{prop}[thm]{Proposition}
\newtheorem{cor}[thm]{Corollary}
\newtheorem{lem}[thm]{Lemma}
\theoremstyle{definition}
\newtheorem{defn}[thm]{Definition}
\newtheorem{example}[thm]{Example}

\newtheorem{remark}[thm]{Remark}

\DeclareMathOperator{\st}{\mathsf{st}}
\DeclareMathOperator{\order}{\mathsf{or}}

\DeclareMathOperator{\funnel}{\mathsf{fun}}

\DeclareMathOperator{\argmin}{\mathsf{argmin}}

\newcommand{\T}{\mathcal{T}}
\newcommand{\F}{\mathcal{F}}
\renewcommand{\L}{\mathcal{L}}

\newcommand{\definition}[1]{\emph{\textbf{#1}}}

\makeatletter
\renewcommand*{\@fnsymbol}[1]{\ensuremath{\ifcase#1\or 1\or 2 \or 3 \or 4 \or 5 \or 6 \or 7 \or 8 \or 9 \or 10
   \else\@ctrerr\fi}}
\makeatother

\begin{document}

\title{Promotion, Tangled Labelings, and Sorting Generating Functions}

\author[M. Bayer]{Margaret Bayer}
\address{Margaret Bayer: Department of Mathematics, University of Kansas, Lawrence, KS, USA} 
\email{bayer@ku.edu}

\author[H. Chau]{Herman Chau}
\address{Herman Chau: Department of Mathematics, University of Washington, Seattle, WA, USA}
\email{hchau@uw.edu}

\author[M. Denker]{Mark Denker}
\address{Mark Denker: Department of Mathematics, University of Kansas, Lawrence, KS, USA}
\email{mark.denker@ku.edu}

\author[O. Goff]{Owen Goff}
\address{Owen Goff: Department of Mathematics, University of Wisconsin, Madison, WI, USA}
\email{ogoff@wisc.edu}

\author[J. Kimble]{Jamie Kimble}
\address{Jamie Kimble: Department of Mathematics, Michigan State University, East Lansing, MI, USA}
\email{kimblej2@msu.edu}

\author[Y. Lee]{Yi-Lin Lee}
\address{Yi-Lin Lee: Department of Mathematics, Indiana University, Bloomington, IN, USA}
\email{yillee@iu.edu}

\author[J. Liang]{Jinting Liang}
\address{Jinting Liang: Department of Mathematics, University of British Columbia, BC, Canada}
\email{liangj@math.ubc.ca}


\begin{abstract}
We study Defant and Kravitz's generalization of Sch\"utzenberger's promotion operator to arbitrary labelings of finite posets in two directions. Defant and Kravitz showed that applying the promotion operator $n-1$ times to a labeling of a poset on $n$ elements always gives a natural labeling of the poset and called a labeling \emph{tangled} if it requires the full $n-1$ promotions to reach a natural labeling. They also conjectured that there are at most $(n-1)!$ tangled labelings for any poset on $n$ elements. 

In the first direction, we propose a further strengthening of their conjecture by partitioning tangled labelings according to the element labeled $n-1$ and prove that this stronger conjecture holds for \emph{inflated rooted forest posets} and a new class of posets called \emph{shoelace posets}. In the second direction, we introduce sorting generating functions and cumulative generating functions for the number of labelings that require $k$ applications of the promotion operator to give a natural labeling. We prove that the coefficients of the cumulative generating function of the ordinal sum of antichains are log-concave and obtain a refinement of the weak order on the symmetric group.
\end{abstract}

\maketitle

\section{Introduction}
\label{section:Introduction}

\subsection{Background}
A partially ordered set $P$ on $n$ elements is \definition{naturally labeled} if each element is labeled with an integer between $1$ and $n$ such that the labels respect the ordering on elements of $P$. In 1972, Sch\"utzenberger introduced the \definition{promotion} operator on natural labelings of posets \cite{schutzenberger-1972}. The motivation for the promotion operator comes from an earlier paper of Sch\"utzenberger \cite{Schutzenberger-1963}, in which he defines a related operator, \definition{evacuation}, to study the celebrated RSK algorithm. Promotion and evacuation were subsequently studied by Stanley in relation to Hecke algebra products \cite{stanley-2009}, by Rhoades in relation to cyclic sieving phenomenons \cite{rhoades-2010}, and by Striker and Williams in relation to rowmotion and alternating sign matrices \cite{strikerwilliams2012}, among many others.

As originally defined, promotion applies only to natural labelings of posets. Defant and Kravitz considered generalizing the notion of promotion to operate on arbitrary poset labelings and referred to their generalization as \definition{extended promotion} \cite{defant-kravitz-2023}. Given a labeling $L$ of a poset, the extended promotion of $L$ is denoted $\partial L$. A key property of extended promotion is that applying it to a labeling yields a new labeling that is ``closer'' to a natural labeling. This property is quantified precisely in the following theorem.

\begin{thm}[{\cite[Theorem~2.8]{defant-kravitz-2023}}]\label{theorem:defant-kravitz-main}
    For any labeling $L$ of an $n$-element poset, the labeling $\partial^{n-1} L$ is a natural labeling.
\end{thm}

When applied to an arbitrary poset labeling, extended promotion will always result in a natural labeling after a maximum of $n-1$ applications. Applied to a natural labeling of a poset, the extended promotion will always produce another natural labeling. Defant and Kravitz \cite{defant-kravitz-2023} define a tangled labeling of an $n$-element poset as a labeling that requires $n-1$ promotions to give a natural labeling. Intuitively, the tangled labelings of a poset are those that are furthest from being ``sorted" by extended promotion; they require the full $n-1$ applications of extended promotion in \cref{theorem:defant-kravitz-main}. Defant and Kravitz studied the number of tangled labelings of a poset and conjectured the following upper bound on the number of tangled labelings.

\begin{conj}[{\cite[Conjecture 5.1]{defant-kravitz-2023}}] \label{conj:main-conjecture}
    An $n$-element poset has at most $(n-1)!$ tangled labelings.
\end{conj}

Defant and Kravitz proved an enumerative formula for a large class of posets known as \definition{inflated rooted forest posets} (see \cref{section:inflated-rooted-forests} for details). This formula was used by Hodges to show \cref{conj:main-conjecture} holds for all inflated rooted forest posets. Furthermore, Hodges conjectured a stronger version of \cref{conj:main-conjecture}.

\begin{conj}[{\cite[Conjecture 31]{hodges-2022}}]
    \label{conjecture:hodges}
An $n$-element poset with $m$ minimal elements has at most $(n-m)(n-2)!$ tangled labelings.
\end{conj}

Both \cite{defant-kravitz-2023} and \cite{hodges-2022} also considered counting labelings by the number of extended promotion steps needed to yield a natural labeling. In the preprint of \cite{defant-kravitz-2023}, Defant and Kravitz proposed the following, listed as Conjecture 5.2 in \cite{defant2020promotionsorting}. Hodges further examined this conjecture.

\begin{conj}[{\cite[Conjecture 29]{hodges-2022}}]\label{false}
Let $P$ be an $n$-element poset, and let $a_k(P)$ denote the number of labelings of $P$ requiring exactly $k$ applications of the extended promotion to be a natural labeling. Then the sequence $a_0(P),\dotsc,a_{n-1}(P)$ is unimodal.
\end{conj}

\subsection{Outline and Summary of Main Results}

In this paper, we study the number of tangled labelings of posets by partitioning tangled labelings according to which poset element has label $n-1$. We propose the following new conjecture.

\begin{conj}[{[The $(n-2)!$ Conjecture]}]
    \label{conj:(n-1)-refinement}
   Let $P$ be an $n$-element poset with $n \ge 2$. For all $x \in P$, let $|\T_x(P)|$ denote the number of tangled labelings of $P$ such that $x$ is labeled $n-1$. Then $|\T_x(P)| \le (n-2)!$ with equality if and only if there is a unique minimal element $y \in P$ such that $y <_P x$.
\end{conj}

By results in \cref{section:promotion-properties}, both \cref{conj:main-conjecture} and \cref{conjecture:hodges} follow from the $(n-2)!$ conjecture. In \cref{theorem:inflated_rooted} and \cref{thm:shoelace}, we prove that that the $(n-2)!$ conjecture holds for inflated rooted forest posets and for a new class of posets that we call \emph{shoelace posets}. Furthermore, the conjecture has been computationally verified on all posets with nine or fewer elements.

Following \cite{hodges-2022}, we also consider the sorting time for labelings that are not tangled and introduce associated generating functions. In \cref{rmk:false}, we give a poset on six elements that is a counterexample to \cref{false}. Our results completely determine the generating functions for ordinal sums of antichains. We introduce a related generating function called the \emph{cumulative generating function} and prove log-concavity of the cumulative generating function for ordinal sums of antichains.

In \cref{section:promotion-properties} we review the basic properties of extended promotion. In \cref{section:inflated-rooted-forests} we prove that inflated rooted forest posets satisfy the $(n-2)!$ conjecture. In \cref{section:shoelace-posets} we prove that inflated shoelace posets satisfy the $(n-2)!$ conjecture and give an exact enumeration for the number of tangled labelings of a particular type of shoelace poset called a $W$-poset. In \cref{section:generating-functions} we study the generating function of the sorting time of labelings of the ordinal sum of a poset $P$ with the antichain $T_k$ on $k$ elements. In \cref{sec:ordinalsum} we show that the cumulative generating function for ordinal sums of antichains are log-concave and use the cumulative generating functions to introduce a new partial order on the symmetric group $\mathfrak{S}_n$. In \cref{section:future-work} we propose future directions to explore.

\section{Definitions and Properties of Extended Promotion}\label{section:promotion-properties}

In this section, we review and prove some properties of the extended promotion operator that will be used in later sections. Many of the definitions and results in this section come from \cite{defant-kravitz-2023} and are cited appropriately.

\subsection{Notation and Terminology}

The notation $[n]$ denotes the set $\{1, 2, \ldots, n\}$. For a partially ordered set (or poset) $P$, the partial order on $P$ will be denoted $<_P$. An element $y \in P$ is said to \definition{cover} $x \in P$, denoted $x \lessdot_P y$, if $x<_P y$ and there does not exist an element $z \in P$ such that $x <_P z <_P y$. A \definition{lower (resp. upper) order ideal} of $P$ is a set $X \subseteq P$ with the property that if $y \in X$ and $x <_P y$ (resp. $x >_P y$) then $x \in X$ also. For an element $y \in P$, the \definition{principal lower order ideal} of $y$ is denoted $\downarrow y = \{x \in P : x \le_P y\}$. A poset $P$ is said to be \definition{connected} if its Hasse diagram is a connected graph. In this paper, we only consider finite posets and assume the reader is familiar with standard results on posets as can be found in \cite[Chapter 3]{ECI}. 

A \definition{labeling} of a poset $P$ with $n$ elements is a bijection from $P$ to $[n]$.  A labeling $L$ of $P$ is a \definition{natural labeling} if the sequence $L^{-1}(1), L^{-1}(2), \ldots, L^{-1}(n)$ is a linear extension of $P$. Equivalently, for any elements $x,y \in P$, if $x<_P y$ then $L(x)<L(y)$. Given a poset $P$, the set of all labelings of $P$ will be denoted $\Lambda(P)$. The set of all natural
labelings (equivalently, linear extensions) of $P$ will be denoted $\mathcal{L}(P)$.

\begin{defn}[{\cite[Definition 2.1]{defant-kravitz-2023}}]
    \label{definition:extended-promotion}
    Let $P$ be an $n$-element poset and $L \in \Lambda(P)$. The \definition{extended promotion of $L$}, denoted $\partial L$, is obtained from $L$ by the following algorithm:
    \begin{enumerate}
        \item Repeat until the element labeled 1 is maximal: Let $x$ be the element labeled 1 and let $y$ be the element with the smallest label such that $y >_P x$. Swap the labels of $x$ and $y$.
        \item Simultaneously replace the label 1 with $n$ and replace the label $i$ with $i-1$ for all $i > 1$.
    \end{enumerate}
\end{defn}

In what follows, we will refer to extended promotion simply as 
\definition{promotion}. For $i \ge 0$, the notations $L_i$ and $\partial^i L$ are used interchangeably to denote the $i$th promotion of $L$. By convention, $L_0$ and $\partial^0 L$ denote the original labeling $L$. Promotion can be loosely thought of as ``sorting" a labeling $L$ so that 
$\partial L$ is closer to being a natural labeling. 
\begin{defn}({\cite[Section 2]{defant-kravitz-2023}}\footnote{We remark that promotion chains were first defined by Stanley \cite{stanley-2009} in the context of promotion on natural labelings.})
Let $L \in \Lambda(P)$. The \definition{promotion chain} of $L$ is the ordered set of elements of $P$ whose labels are swapped in the first step of \cref{definition:extended-promotion}. The order of the promotion chain is the order in which the labels were swapped in the first step of \cref{definition:extended-promotion}.
\end{defn}

\begin{example}\label{example:promotion}
\cref{figure:one-promotion} shows the promotion algorithm applied to a labeling $L$ of a 6-element poset $P$. The promotion chain of $L$ is the ordered sequence $L^{-1}(1), L^{-1}(2), L^{-1}(5)$. A sequence of five promotions of $L$ is shown in \cref{figure:full-promotion}. Observe that $L_i$ is not a natural labeling for $0 \le i < 5$ but $L_5$ is a natural labeling.
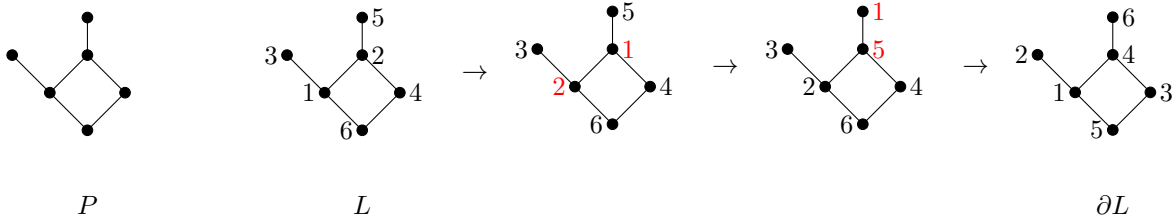
\begin{figure}[htb!]
\begin{center}
     \begin{tikzpicture}
        \draw (0,0) -- (-0.5,0.5) -- (0,1) -- (0.5,0.5) -- (0,0);
        \draw (-0.5,0.5) -- (-1,1);
        \draw (0,1) -- (0,1.5);
        \filldraw[black] (0,0) circle (2pt) node [anchor=east] {};
        \filldraw[black] (-0.5,0.5) circle (2pt) node [anchor=east] {};
        \filldraw[black] (0,1) circle (2pt) node [anchor=west] {};
        \filldraw[black] (0.5,0.5) circle (2pt) node [anchor=west] {};
        \filldraw[black] (-1,1) circle (2pt) node [anchor=east] {};
        \filldraw[black] (0,1.5) circle (2pt) node [anchor=west] {};
        \node at (0, -1) {$P$};
        \node at (2,0.75) {};
    \end{tikzpicture}
    \begin{tikzpicture}
        \draw (0,0) -- (-0.5,0.5) -- (0,1) -- (0.5,0.5) -- (0,0);
        \draw (-0.5,0.5) -- (-1,1);
        \draw (0,1) -- (0,1.5);
        \filldraw[black] (0,0) circle (2pt) node [anchor=east] {$6$};
        \filldraw[black] (-0.5,0.5) circle (2pt) node [anchor=east] {$1$};
        \filldraw[black] (0,1) circle (2pt) node [anchor=west] {$2$};
        \filldraw[black] (0.5,0.5) circle (2pt) node [anchor=west] {$4$};
        \filldraw[black] (-1,1) circle (2pt) node [anchor=east] {$3$};
        \filldraw[black] (0,1.5) circle (2pt) node [anchor=west] {$5$};
        \node at (0, -1) {$L$};
        \node at (1.5,0.75) {$\rightarrow$};
    \end{tikzpicture}
    \begin{tikzpicture}
        \draw (0,0) -- (-0.5,0.5) -- (0,1) -- (0.5,0.5) -- (0,0);
        \draw (-0.5,0.5) -- (-1,1);
        \draw (0,1) -- (0,1.5);
        \filldraw[black] (0,0) circle (2pt) node [anchor=east] {$6$};
        \filldraw[black] (-0.5,0.5) circle (2pt) node [anchor=east] {$\textcolor{red}{2}$};
        \filldraw[black] (0,1) circle (2pt) node [anchor=west] {$\textcolor{red}{1}$};
        \filldraw[black] (0.5,0.5) circle (2pt) node [anchor=west] {$4$};
        \filldraw[black] (-1,1) circle (2pt) node [anchor=east] {$3$};
        \filldraw[black] (0,1.5) circle (2pt) node [anchor=west] {$5$};
        \node at (0, -1.2) {};
        \node at (1.5,0.75) {$\rightarrow$};
    \end{tikzpicture}
    \begin{tikzpicture}
        \draw (0,0) -- (-0.5,0.5) -- (0,1) -- (0.5,0.5) -- (0,0);
        \draw (-0.5,0.5) -- (-1,1);
        \draw (0,1) -- (0,1.5);
        \filldraw[black] (0,0) circle (2pt) node [anchor=east] {$6$};
        \filldraw[black] (-0.5,0.5) circle (2pt) node [anchor=east] {$2$};
        \filldraw[black] (0,1) circle (2pt) node [anchor=west] {$\textcolor{red}{5}$};
        \filldraw[black] (0.5,0.5) circle (2pt) node [anchor=west] {$4$};
        \filldraw[black] (-1,1) circle (2pt) node [anchor=east] {$3$};
        \filldraw[black] (0,1.5) circle (2pt) node [anchor=west] {$\textcolor{red}{1}$};
        \node at (0, -1.2) {};
        \node at (1.5,0.75) {$\rightarrow$};
    \end{tikzpicture}
    \begin{tikzpicture}
        \draw (0,0) -- (-0.5,0.5) -- (0,1) -- (0.5,0.5) -- (0,0);
        \draw (-0.5,0.5) -- (-1,1);
        \draw (0,1) -- (0,1.5);
        \filldraw[black] (0,0) circle (2pt) node [anchor=east] {$5$};
        \filldraw[black] (-0.5,0.5) circle (2pt) node [anchor=east] {$1$};
        \filldraw[black] (0,1) circle (2pt) node [anchor=west] {$4$};
        \filldraw[black] (0.5,0.5) circle (2pt) node [anchor=west] {$3$};
        \filldraw[black] (-1,1) circle (2pt) node [anchor=east] {$2$};
        \filldraw[black] (0,1.5) circle (2pt) node [anchor=west] {$6$};
        \node at (0, -1) {$\partial L$};
    \end{tikzpicture}
    \caption{One promotion of the labeling $L$ on poset $P$.  Swapped labels are shown in red.}

    \label{figure:one-promotion}
    \end{center}
\end{figure}

\begin{defn}[{\cite[Section 1.1]{defant-kravitz-2023}}]\label{def.order}
    Let $P$ be an $n$-element poset and $L \in \Lambda(P)$. The \definition{order} or \definition{sorting time} of $L$, denoted $\order(L)$, is the smallest integer $k \ge 0$ such that $L_k \in \L(P)$. If $\order(L)=n-1$, then $L$ is a \definition{tangled labeling}. The set of all tangled labelings of $P$ is denoted $\T(P)$.
\end{defn}
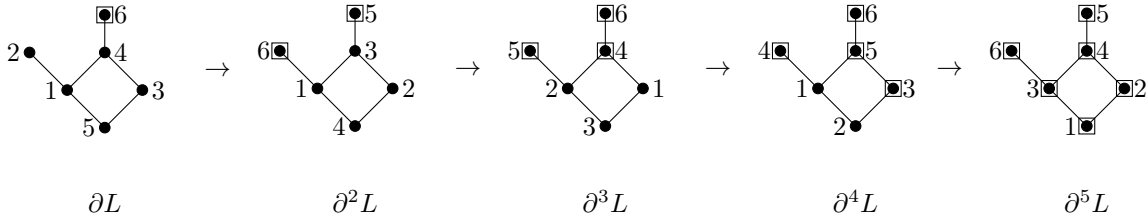
\begin{figure}[htb!]
\begin{center}
    \begin{tikzpicture}
        \draw (0,0) -- (-0.5,0.5) -- (0,1) -- (0.5,0.5) -- (0,0);
        \draw (-0.5,0.5) -- (-1,1);
        \draw (0,1) -- (0,1.5);
        \filldraw[black] (0,0) circle (2pt) node [anchor=east] {$5$};
        \filldraw[black] (-0.5,0.5) circle (2pt) node [anchor=east] {$1$};
        \filldraw[black] (0,1) circle (2pt) node [anchor=west] {$4$};
        \filldraw[black] (0.5,0.5) circle (2pt) node [anchor=west] {$3$};
        \filldraw[black] (-1,1) circle (2pt) node [anchor=east] {$2$};
        \filldraw[black] (0,1.5) circle (2pt) node [anchor=west] {6};
        \node[draw, minimum size=0.5em, inner sep=3pt] at (0,1.5) { };
        \node at (0, -1) {$\partial L$};
        \node at (1.5,0.75) {$\rightarrow$};
    \end{tikzpicture}
    \begin{tikzpicture}
        \draw (0,0) -- (-0.5,0.5) -- (0,1) -- (0.5,0.5) -- (0,0);
        \draw (-0.5,0.5) -- (-1,1);
        \draw (0,1) -- (0,1.5);
        \filldraw[black] (0,0) circle (2pt) node [anchor=east] {$4$};
        \filldraw[black] (-0.5,0.5) circle (2pt) node [anchor=east] {$1$};
        \filldraw[black] (0,1) circle (2pt) node [anchor=west] {$3$};
        \filldraw[black] (0.5,0.5) circle (2pt) node [anchor=west] {$2$};
        \filldraw[black] (-1,1) circle (2pt) node [anchor=east] {$6$};
        \filldraw[black] (0,1.5) circle (2pt) node [anchor=west] {$5$};
        \node[draw, minimum size=0.5em, inner sep=3pt] at (0,1.5) { };
        \node[draw, minimum size=0.5em, inner sep=3pt] at (-1,1) { };
        \node at (0, -1) {$\partial^2 L$};
        \node at (1.5,0.75) {$\rightarrow$};
    \end{tikzpicture}
    \begin{tikzpicture}
        \draw (0,0) -- (-0.5,0.5) -- (0,1) -- (0.5,0.5) -- (0,0);
        \draw (-0.5,0.5) -- (-1,1);
        \draw (0,1) -- (0,1.5);
        \filldraw[black] (0,0) circle (2pt) node [anchor=east] {$3$};
        \filldraw[black] (-0.5,0.5) circle (2pt) node [anchor=east] {$2$};
        \filldraw[black] (0,1) circle (2pt) node [anchor=west] {$4$};
        \filldraw[black] (0.5,0.5) circle (2pt) node [anchor=west] {$1$};
        \filldraw[black] (-1,1) circle (2pt) node [anchor=east] {$5$};
        \filldraw[black] (0,1.5) circle (2pt) node [anchor=west] {$6$};
        \node[draw, minimum size=0.5em, inner sep=3pt] at (0,1.5) { };
        \node[draw, minimum size=0.5em, inner sep=3pt] at (-1,1) { };
        \node[draw, minimum size=0.5em, inner sep=3pt] at (0,1) { };
        \node at (0, -1) {$\partial^3 L$};
        \node at (1.5,0.75) {$\rightarrow$};
    \end{tikzpicture}
    \begin{tikzpicture}
        \draw (0,0) -- (-0.5,0.5) -- (0,1) -- (0.5,0.5) -- (0,0);
        \draw (-0.5,0.5) -- (-1,1);
        \draw (0,1) -- (0,1.5);
        \filldraw[black] (0,0) circle (2pt) node [anchor=east] {$2$};
        \filldraw[black] (-0.5,0.5) circle (2pt) node [anchor=east] {$1$};
        \filldraw[black] (0,1) circle (2pt) node [anchor=west] {$5$};
        \filldraw[black] (0.5,0.5) circle (2pt) node [anchor=west] {$3$};
        \filldraw[black] (-1,1) circle (2pt) node [anchor=east] {$4$};
        \filldraw[black] (0,1.5) circle (2pt) node [anchor=west] {$6$};
        \node[draw, minimum size=0.5em, inner sep=3pt] at (0,1.5) { };
        \node[draw, minimum size=0.5em, inner sep=3pt] at (-1,1) { };
        \node[draw, minimum size=0.5em, inner sep=3pt] at (0,1) { };
        \node[draw, minimum size=0.5em, inner sep=3pt] at (0.5,0.5) { };
        \node at (0, -1) {$\partial^4 L$};
        \node at (1.25,0.75) {$\rightarrow$};
    \end{tikzpicture}
    \begin{tikzpicture}
        \draw (0,0) -- (-0.5,0.5) -- (0,1) -- (0.5,0.5) -- (0,0);
        \draw (-0.5,0.5) -- (-1,1);
        \draw (0,1) -- (0,1.5);
        \filldraw[black] (0,0) circle (2pt) node [anchor=east] {$1$};
        \filldraw[black] (-0.5,0.5) circle (2pt) node [anchor=east] {$3$};
        \filldraw[black] (0,1) circle (2pt) node [anchor=west] {$4$};
        \filldraw[black] (0.5,0.5) circle (2pt) node [anchor=west] {$2$};
        \filldraw[black] (-1,1) circle (2pt) node [anchor=east] {$6$};
        \filldraw[black] (0,1.5) circle (2pt) node [anchor=west] {$5$};
        \node[draw, minimum size=0.5em, inner sep=3pt] at (0,1.5) { };
        \node[draw, minimum size=0.5em, inner sep=3pt] at (-1,1) { };
        \node[draw, minimum size=0.5em, inner sep=3pt] at (0,1) { };
        \node[draw, minimum size=0.5em, inner sep=3pt] at (0.5,0.5) { };
        \node[draw, minimum size=0.5em, inner sep=3pt] at (-0.5,0.5) { };
        \node[draw, minimum size=0.5em, inner sep=3pt] at (0,0) { };
        \node at (0, -1) {$\partial^5 L$};
    \end{tikzpicture}
    \caption{Promotions of the labeling $L$ in \cref{figure:one-promotion}. Elements enclosed in a box are frozen.} 
    \label{figure:full-promotion}      
    \end{center}
\end{figure}
\end{example}
\begin{defn}
Let $P$ be an $n$-element poset and $x \in P$. A labeling $L$ of $P$ is said to be an \definition{$x$-labeling} if $L(x) = n-1$. The set of all tangled $x$-labelings of $P$ is denoted $\T_x(P)$.
\end{defn}

For a poset $P$, the set of tangled labelings $\T(P)$ is the disjoint union of $\T_x(P)$ as $x$ ranges over elements in $P$. Thus, the number of tangled labelings of $P$ is equal to the sum
\begin{equation}
    \label{equation:tangled-refinement}
    |\T(P)| = \sum_{x \in P} |\T_x(P)|.
\end{equation}
It readily follows from \cref{equation:tangled-refinement} that the $(n-2)!$ conjecture implies \cref{conj:main-conjecture}.

\begin{defn}[{\cite[Section 1.3]{defant-kravitz-2023}}]
\label{standard-def}
Let $P$ be an $n$-element poset, $Q$ be an $m$-element subposet of $P$, and $L \in \Lambda(P)$. The \definition{standardization} of $L$ on $Q$ is the unique labeling $\st(L): Q \to [m]$ such that $\st(L)(x) < \st(L)(y)$ if and only if $L(x) < L(y)$ for all $x,y \in Q$.
\end{defn} 

\begin{defn}[{\cite[Section 2]{defant-kravitz-2023}}]
Let $P$ be an $n$-element poset and $x \in P$. The element $x$ is said to be \definition{frozen} with respect to a labeling $L \in \Lambda(P)$ if $L^{-1}(\{a, a+1, \ldots, n\})$ is an upper order ideal for every $a$ such that $L(x) \le a \le n$. The set of frozen elements of $L$ will be denoted $\F(L)$.
\end{defn}

Equivalently, if $x$ is frozen, then the standardization of $L$ on the subposet $L^{-1}(\{L(x), L(x)+1, \ldots, n\})$ is a natural labeling. Thus, $L$ is a natural labeling of $P$ if and only if $\F(L)=P$.
Observe that by \cref{definition:extended-promotion}, for any labeling $L$ of an $n$-element poset $P$, the element labeled $n$ in $\partial L$ is a maximal element of $P$.
More  generally, $L^{-1}_{j+1}(n-j)$ is frozen, so the elements of $P$ with labels
$\{n-j,n-j+1,\ldots, n\}$ are ``sorted.'' The standardization of $L_{j+1}$ on the subposet of $P$ whose elements have $L_{j+1}$-labels in
$\{n-j,n-j+1,\ldots, n\}$ is a natural labeling.

\begin{example}
    In \cref{figure:full-promotion}, the frozen elements of each labeling are enclosed in boxes. Observe that once an element is frozen, it remains frozen in subsequent promotions. \cref{figure:standardization} shows a subposet $Q$ and the standardization of the labeling $L$ in \cref{figure:one-promotion} on $Q$.
    \begin{figure}[!h]
        \centering
        \begin{tikzpicture}
            \draw (0,0) -- (-0.5,0.5) -- (0,1) -- (0.5,0.5) -- (0,0);
            \draw (-0.5,0.5) -- (-1,1);
            \draw (0,1) -- (0,1.5);
            \draw[dashed] (-0.9, -0.3) -- (-0.9, 1.3) -- (0.9, 1.3) -- (0.9, -0.3) -- (-0.9, -0.3);
            \filldraw[black] (0,0) circle (2pt) node [anchor=east] {$4$};
            \filldraw[black] (-0.5,0.5) circle (2pt) node [anchor=east] {$1$};
            \filldraw[black] (0,1) circle (2pt) node [anchor=west] {$2$};
            \filldraw[black] (0.5,0.5) circle (2pt) node [anchor=west] {$3$};
            \filldraw[black] (-1,1) circle (2pt) node [anchor=east] {};
            \filldraw[black] (0,1.5) circle (2pt) node [anchor=west] {};
        \end{tikzpicture}
        \caption{The standardization of the labeling $L$ in \cref{figure:one-promotion} on the subposet in the dotted box.}
        \label{figure:standardization}
    \end{figure}
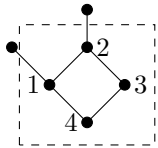
\end{example}

We conclude this subsection by introducing funnels and basins. The basin elements of a poset are a subset of its minimal elements. In \cref{corollary:tangled-labeling-n-(n-1)-locations}, we will see that for tangled labelings, basins are the appropriate subset of minimal elements to pay attention to.

\begin{defn}
Let $x \in P$ be a minimal element. The  \definition{funnel} of $x$ is
\[
    \funnel(x) = \{y \in P : \text{$x <_P y$ and $x$ is the unique minimal element in $\downarrow y$}\}.
\]
\end{defn}

\begin{defn}
    A minimal element $x \in P$ is a \definition{basin} if $\funnel(x) \neq \varnothing$.
\end{defn}

\begin{example}
Let $P$ be the poset with the labeling $L$ in \cref{fig:funnel}. The basin elements in $P$ are $g$ and $i$. Their funnels are $\funnel(g) = \{d\}$ and $\funnel(i) = \{f, c\}$, respectively. There are two basins $g, i$ in the lower order ideal $\downarrow a$ and a single basin $g$ in the lower order ideal $\downarrow b$.
\end{example}
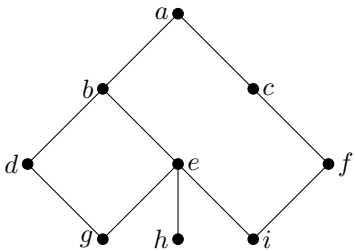
\begin{figure}[ht]
\centering
\begin{tikzpicture}
            \draw (0,0) -- (0,1) -- (-1,0) -- (-2,1) -- (0,3) -- (2,1) -- (1,0) -- (0,1);
            \draw (0,1) -- (-1,2);
            \filldraw[black] (0,0) circle (2pt) node [anchor=east] {$h$};
            
            \filldraw[black] (0,1) circle (2pt) node [anchor=west] {$e$};
            \filldraw[black] (0,3) circle (2pt) node [anchor=east] {$a$};
            \filldraw[black] (-1,0) circle (2pt) node [anchor=east] {$g$};
            \filldraw[black] (-1,2) circle (2pt) node [anchor=east] {$b$};
            \filldraw[black] (-2,1) circle (2pt) node [anchor=east] {$d$};
            \filldraw[black] (1,0) circle (2pt) node [anchor=west] {$i$};
            \filldraw[black] (1,2) circle (2pt) node [anchor=west] {$c$};
            \filldraw[black] (2,1) circle (2pt) node [anchor=west] {$f$};
        \end{tikzpicture}
\caption{A poset with two basin elements $g$ and $i$.}
\label{fig:funnel}
\end{figure}

In the terminology of this section, Defant's and Kravitz's characterization of tangled labelings is as follows.

\begin{thm}[{\cite[Theorem~2.10]{defant-kravitz-2023}}] A poset $P$ has a tangled labeling if and only if $P$ has a basin.
\end{thm}

\subsection{Properties of Extended Promotion}
In this subsection, we provide some general lemmas on extended promotion and tangled labelings. We begin with a lemma implicit in  \cite{defant-kravitz-2023} that gives a useful criterion for checking whether or not a labeling is tangled.

\begin{lem}
\label{lemma:tangled-labelings-characterization}
    Let $P$ be a poset on $n$ elements and $L \in \Lambda(P)$. The labeling $L$ is tangled if and only if both of the following conditions are met:
    \begin{enumerate}
        \item $L^{-1}(n)$ is minimal in $P$,
        \item $L^{-1}(n) <_P L_{n-2}^{-1}(1)$.
    \end{enumerate}
\end{lem}
\begin{proof}
    First, we will prove that conditions $(1)$ and $(2)$ together are sufficient for $L$ to be tangled. Let $x$ denote $L^{-1}(n)$. By condition $(1)$, $x$ is minimal so $L_{i+r}(x) = L_{i} (x) - r$ whenever $L_{i}(x) > r$. Since $L(x) = n$, it follows that $L_{n-2}(x) = 2$ and hence $L_{n-2}^{-1}(2) = L^{-1}(n)$. Substituting into condition $(2)$ yields $ L_{n-2}^{-1}(2) <_P L_{n-2}^{-1}(1)$. Thus, $L_{n-2}$ is not yet sorted, and so $L$ is tangled.

    By \cite[Lemma 3.8]{defant-kravitz-2023}, condition $(1)$ is necessary for $L$ to be tangled. Thus, it remains to show that condition $(2)$ follows from assuming that $L$ is tangled and that condition $(1)$ holds. By \cite[Lemma 2.7]{defant-kravitz-2023}, $L_{n-2}^{-1}(3), \ldots, L_{n-2}^{-1}(n)$ are frozen with respect to $L_{n-2}$. Since $L$ is tangled, $L_{n-2}$ is not sorted, which may occur only if $L_{n-2}^{-1}(2) <_P L_{n-2}^{-1}(1)$. Because $L^{-1}(n)$ is minimal, we may substitute $L^{-1}(n) = L_{n-2}^{-1}(2)$ to yield condition $(2)$.
\end{proof}

As a consequence of (2), the element labeled $n-1$ cannot be minimal in a tangled labeling of $P$. If an $n$-element poset $P$ has $m$ minimal elements, then \cref{conj:(n-1)-refinement} would imply that the number of tangled labelings of $P$ is at most $(n-m)(n-2)!$. Therefore, \cref{conj:(n-1)-refinement} implies \cref{conjecture:hodges} and hence \cref{conj:main-conjecture}.

\begin{lem}
    \label{lemma:labels-slide-down}
    Let $P$ be a poset on $n$ elements and $L \in \Lambda(P)$. Then for all $2 \le i \le n$ and $0 \le j \le n-1$,
    \[L_{j+1}^{-1}(i-1) \le_P L_{j}^{-1}(i).\]
\end{lem}
\begin{proof}
If $i$ is not the label of an element in the promotion chain of $L_j$, then the element $L_j^{-1}(i)$ will be labeled $i-1$ in $L_{j+1}$, so $L_{j+1}^{-1}(i-1) = L_j^{-1}(i)$. If $L_j^{-1}(i)$ is in the promotion chain of $L_j$, let $x$ denote the element immediately preceding $L_j^{-1}(i)$ in the promotion chain of $L_j$. Such an element exists since $i \ge 2$ so $L_j^{-1}(i)$ cannot be the first element in the promotion chain. It follows that $L_{j+1}^{-1}(i-1) = x \le_P L_j^{-1}(i)$.
\end{proof}

A consequence of \cref{lemma:labels-slide-down} is that for all $2\le i\le n$,
    \begin{equation}
        L_{i-1}^{-1}(1)\le_P\ldots\le_P L_1^{-1}(i-1)\le_P L^{-1}(i).
    \end{equation}
    Setting $i=n-1$ gives, in particular,\begin{equation}\label{Inequals for label n-1}
    L_{n-2}^{-1}(1)\leq_P L_{n-3}^{-1}(2) \leq_P \ldots \leq_P L_1^{-1}(n-2)\leq_P L^{-1}(n-1).
\end{equation}

\begin{cor} 
\label{cor:(n-r)-inequality-when-tangled}
    Let $P$ be a poset on $n$ elements and let $L \in \L(P)$ be a tangled labeling. For $r = 0, 1, \ldots, n-2$,
    \begin{equation*}
        L_r^{-1}(n-r) <_P L_r^{-1}(n-1-r).
    \end{equation*}
    In particular, $L^{-1}(n) <_P L^{-1}(n-1)$.
\end{cor}
\begin{proof}
By \cref{lemma:labels-slide-down}, $L_r^{-1}(n-r) \le_P L^{-1}(n)$, and by (2) in Lemma \ref{lemma:tangled-labelings-characterization},  $L^{-1}(n) <_P L_{n-2}^{-1}(1)$.
Additionally, by \cref{Inequals for label n-1}, $L_{n-2}^{-1}(1) \le_P L_{n-3}^{-1}(2) \le_P \cdots \le_P L_{r}^{-1}(n-1-r)$. Combining these inequalities yields the desired result $L_r^{-1}(n-r) <_P L_r^{-1}(n-1-r)$. If we set $r=0$, then we see that $L^{-1}(n) <_P L^{-1}(n-1)$.
\end{proof}

In \cite[Corollary 3.7]{defant-kravitz-2023}, Defant and Kravitz showed that any poset with a unique minimal element satisfies \cref{conj:main-conjecture}. We strengthen this result to show that posets with any number of minimal elements---but only one basin---also satisfy \cref{conj:main-conjecture}. We will need the following lemma that is the key tool in Defant and Kravitz's proof of \cref{theorem:defant-kravitz-main}. 
\begin{lem}[{\cite[Lemma~2.6]{defant-kravitz-2023}}]
   \label{lemma:frozen-set} 
    Let $P$ be an $n$-element poset and let $L \in \Lambda(P) \setminus \L(P)$. Then $\F(L) \subsetneq \F(\partial L)$.
\end{lem}
\begin{prop} 
\label{corollary:tangled-labeling-n-(n-1)-locations}
    If $L$ is a tangled labeling of $P$, then $L^{-1}(n)$ is a basin.  In particular, if $P$ has exactly one basin, then $|\T(P)| \le (n-1)!$.
\end{prop}
\begin{proof}
We first show that
for any minimal element $x\in P$ that is not a basin, there is no tangled labeling $L$ with $L(x)=n$. Suppose to the contrary that there exists such a tangled labeling $L$. 
Let $w = L_{n-2}^{-1}(1)$. By \cref{lemma:tangled-labelings-characterization}, $x <_P w$. Since $x$ is not a basin, $\funnel(x) =\varnothing$. Hence, there exists a minimal element $z \neq x$ such that $z<_P w$.

Since $w=L_{n-2}^{-1}(1)$ and $x=L^{-1}(n) = L_{n-2}^{-1}(2)$, it follows that $z = L_{n-2}^{-1}(m)$ for some $m\geq 3$. The elements $L_{n-2}^{-1}(3),\ldots,L_{n-2}^{-1}(n)$ are frozen as a consequence of \cref{lemma:frozen-set}. Recall that the set of frozen elements is an upper order ideal. Since $z$ is a frozen element and $z <_P w$, $w$ must also be a frozen element, which is a contradiction since $L_{n-2}$ is not a natural labeling. Therefore if $L$ is a tangled labeling and $L^{-1}(n)$ is a minimal element of $P$, then $L^{-1}(n)$ must be a basin.

Finally, suppose $P$ has a unique basin $x$. Then any tangled labeling $L$ of $P$ must satisfy $L(x)=n$. There are $(n-1)!$ labelings $L$ that satisfy $L(x) = n$, so $|\T(P)| \le (n-1)!$.
\end{proof}

The following two lemmas relate tangled labelings and funnels of posets. They will be used in \cref{section:shoelace-posets} to prove that shoelace posets satisfy the $(n-2)!$ conjecture.
\begin{lem} 
    \label{lemma:sufficient-condition-tangled}
    Let $x$ be a basin of $P$ and let $L$ be a labeling such that $L^{-1}(n) = x$ and $L^{-1}(n-1) \in \funnel(x)$. Then $L$ is tangled.
\end{lem}
\begin{proof}
It is clear from the definition of basins that condition $(1)$ of \cref{lemma:tangled-labelings-characterization} is satisfied. So it suffices to show that $L^{-1}(n) <_P (L_{n-2})^{-1}(1)$. From \cref{Inequals for label n-1}
and the condition that $L^{-1}(n-1) \in \funnel(x)$,
\[
x\leq_P (L_{n-2})^{-1}(1)\leq_P  L^{-1}(n-1).
\]
Furthermore,
\[
x=L^{-1}(n)=(L_{n-2})^{-1}(2)\not=(L_{n-2})^{-1}(1).
\]
Thus, we have the strict inequality $x = L^{-1}(n) <_P (L_{n-2})^{-1}(1)$, which is precisely condition $(2)$ of \cref{lemma:tangled-labelings-characterization}.
\end{proof}
\begin{lem}
    \label{lemma:tangled-funnel-condition}
    Let $P$ be a poset on $n$ elements and $L$ a tangled labeling of $P$. Let $x,y \in P$ such that $x$ is a minimal element and $x <_P y$. If $L(x) = n$ and $L(y) = n-1$, then there exists $z \in \funnel(x)$ such that $z \le_P y$.
\end{lem}
\begin{proof}
    Let $z = L_{n-2}^{-1}(1)$. By
\cref{Inequals for label n-1}, $z = L_{n-2}^{-1}(1)  \le_P L^{-1}(n-1) = y$. Thus, $z \le_P y$. Since $L$ is a tangled labeling, \cref{lemma:tangled-labelings-characterization} implies that $x = L_{n-2}^{-1}(2) <_P L_{n-2}^{-1}(1) = z$. There are at least $n-2$ frozen elements with respect to $L_{n-2}$, but $x$ and $z$ are not frozen with respect to $L_{n-2}$. Since the set of frozen elements with respect to a labeling form an upper order ideal, it follows that $z$ covers $x$ and no other elements. Hence, $z \in \funnel(x)$.
\end{proof}
\begin{lem} 
    \label{lemma:n-2-conjecture-disjoint}
    Let $P_1$ be a poset with $n_1$ elements and $P_2$ a poset with $n_2$ elements. If \cref{conj:(n-1)-refinement} holds for $P_1$ and $P_2$, then \cref{conj:(n-1)-refinement} also holds for the disjoint union $P_1 \sqcup P_2$. 
\end{lem}
\begin{proof}
    Let $x \in P_1 \sqcup P_2$ and $L$ be an $x$-labeling of $P$ (i.e., $L(x)=n-1$). If $x \in P_1$ and $n_1 \ge 2$, then by \cite[Theorem 3.4]{defant-kravitz-2023}, $L$ is tangled if and only if $L^{-1}(n) \in P_1$ and $\st(L|_{P_1}) \in \T(P_1)$. Thus, the tangled $x$-labelings of $P_1 \sqcup P_2$ are enumerated by a choice of one of the $|\T_x(P_1)|$ tangled $x$-labelings of $P_1$, one of the $\binom{n_1+n_2-2}{n_1-2}$ assignments of the labels $L^{-1}(P_1) \setminus \{n, n-1\}$, and one of the $n_2!$ labelings on $P_2$. Since $P_1$ satisfies \cref{conj:(n-1)-refinement}, $|\T_x(P_1)| \le (n_1-2)!$. Therefore,
    \begin{equation}
        \label{equation:n-2-conjecture-disjoint}
        \begin{aligned}
        |\T_x(P_1 \sqcup P_2)| &= |\T_x(P_1)| \cdot n_2! \cdot \binom{n_1+n_2-2}{n_1-2}\\
        &\le (n_1-2)! \cdot n_2! \cdot \binom{n_1+n_2-2}{n_1-2}\\
        &= (n_1+n_2-2)!.
        \end{aligned}
    \end{equation}
    
    If $x \in P_1$ and $n_1 < 2$, then by the contrapositive of \cref{cor:(n-r)-inequality-when-tangled}, $L$ is not tangled. In this case, tangled $x$-labelings of $P_1 \sqcup P_2$ do not exist, so $|\T_x(P_1 \sqcup P_2)| \leq (n_1+n_2-2)!$ clearly. Equality in \cref{equation:n-2-conjecture-disjoint} holds if and only if $|\T_x(P_1)| = (n_1-2)!$. Since $P_1$ satisfies \cref{conj:(n-1)-refinement}, $|\T_x(P_1)| = (n_1-2)!$ if and only if there is a unique minimal element $z \in P_1$ such that $z <_{P_1} x$. It follows that equality in \cref{equation:n-2-conjecture-disjoint} holds if and only if there is a unique minimal element $z \in P_1 \sqcup P_2$ such that $z <_{P_1 \sqcup P_2} x$. If $x \in P_2$, then by an identical argument, $|\T_x(P_1 \sqcup P_2)| \le (n_1+n_2-2)!$, with equality if and only if there is a unique minimal element $z \in P_2$ such that $z <_{P_1 \sqcup P_2} x$. Therefore, $P_1 \sqcup P_2$ satisfies \cref{conj:(n-1)-refinement}.
\end{proof}

By \cref{lemma:n-2-conjecture-disjoint}, it suffices to show the $(n-2)!$ conjecture for connected posets. Thus, for the remainder of the paper, we will assume our posets are connected.

\section{Inflated Rooted Forest Posets}
\label{section:inflated-rooted-forests}

In \cite{defant-kravitz-2023}, a large class of posets known as \emph{inflated rooted forest posets} was introduced and it was shown in \cite{hodges-2022} that \cref{conj:main-conjecture} holds for inflated rooted forest posets. In this section, we strengthen this result by showing that \cref{conj:(n-1)-refinement} holds for inflated rooted forest posets.

\begin{defn}[{\cite[Definition 3.2]{defant-kravitz-2023}}]
    Let $P,Q$ be finite posets. The poset $P$ is an \definition{inflation} of $Q$ if there exists a surjective map $\varphi: P \to Q$ that satisfies the following two properties:
    \begin{enumerate}
        \item For any $x \in Q$, the preimage $\varphi^{-1}(x)$ has a unique minimal element in $P$. 
        \item For any $x, y\in P$ such that $\varphi(x) \neq\varphi(y)$,  $x<_P y$ if and only if $\varphi(x) <_Q\varphi(y)$.
    \end{enumerate}
    Such a map $\varphi$ is called an \definition{inflation map}.
    \label{defn:inflation}
\end{defn}

\begin{example}
 In \cref{fig:InflatedRooted} the poset $P$ is an inflation of the poset $Q$. The inflation map $\varphi$ is constant on each colored box in $P$ and maps to the corresponding element in $Q$ pointed to by the arrow. For example, the element labeled $u_{1,1}$ in $Q$ corresponds to the subposet $\varphi^{-1}(u_{1,1})$ in $P$ outlined in green. In general, the preimage of an element in $Q$ must have a unique minimal element, by definition, but may have multiple maximal elements.
\end{example}
\begin{figure}[!hbt]
    \centering
    \includegraphics[width=0.9\textwidth]{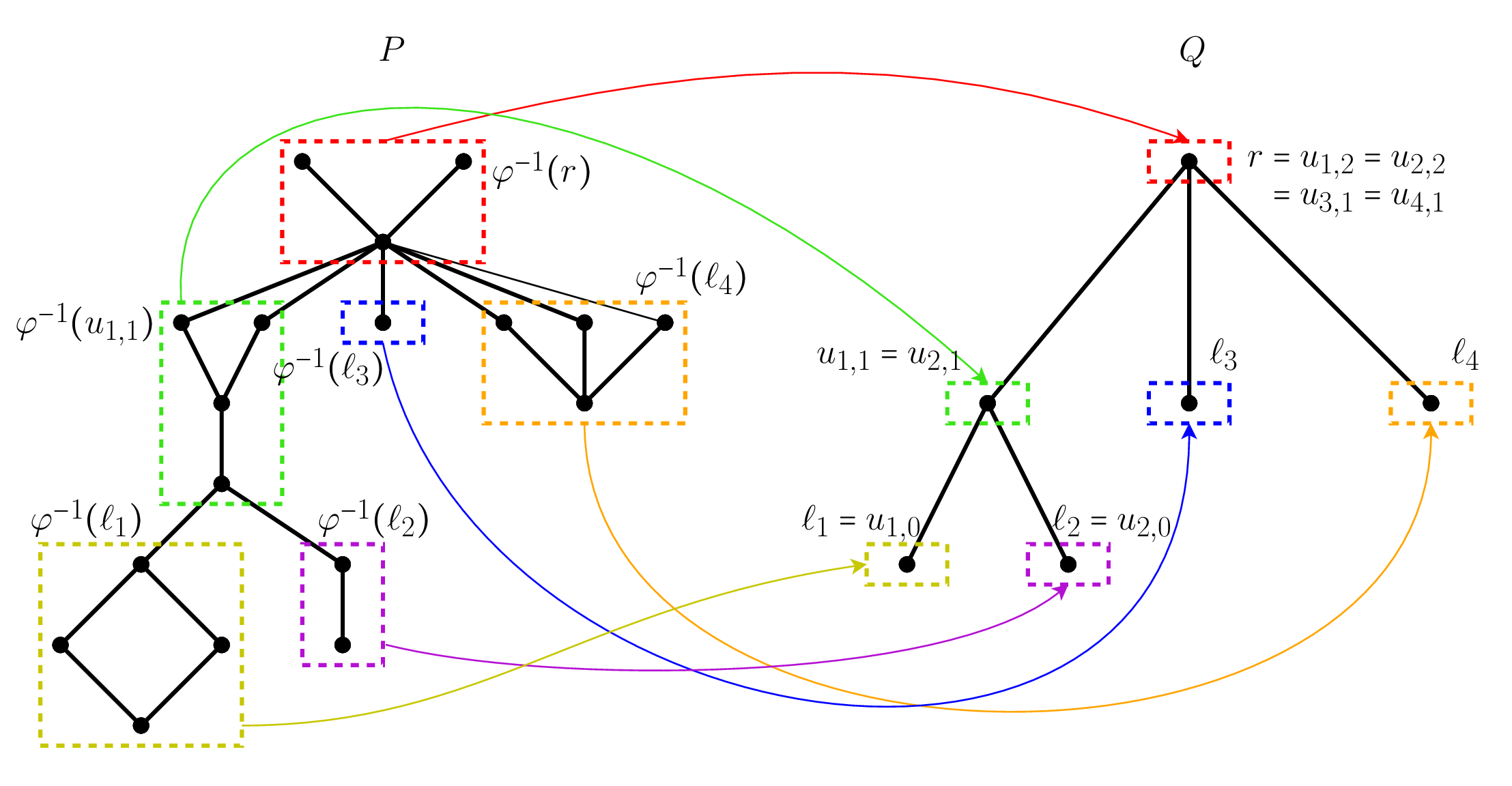}
    \caption{A rooted tree $Q$ and its inflation $P$. The inflation map $\varphi$ is represented by the arrows from $P$ to $Q$.}
    \label{fig:InflatedRooted}
\end{figure}
\begin{defn}
[{\cite[Definition 3.1]{defant-kravitz-2023}}]
    A \definition{rooted tree poset} $Q$ is a finite poset satisfying the following two properties:
    \begin{enumerate}
        \item There is a unique maximal element of $Q$ called the \definition{root} of $Q$.
        \item Every non-root element in $Q$ is covered by exactly one element.
    \end{enumerate}
\end{defn}

A \definition{rooted forest poset} is defined to be a finite poset that can be written as the disjoint union of rooted tree posets. The posets $P$ and $Q$ in \cref{fig:InflatedRooted} are examples of an inflated rooted tree poset and a rooted tree poset, respectively. Throughout the rest of this section, unless otherwise specified, $Q$ will denote a rooted tree poset and $P$ will denote an inflation of $Q$ with inflation map $\varphi$.

The following definitions on inflated rooted tree posets can be found in \cite{defant-kravitz-2023}. We reproduce them here for the reader's convenience and to state \cref{lemma:inflation-inequality-lemma} precisely. Let $r$ be the root of $Q$ and let $x$ be a non-root element of $Q$. The unique element $y$ that covers $x$ in $Q$ is called the \definition{parent} of $x$. The minimal elements of $Q$ are called \definition{leaves}. A rooted tree poset is said to be \definition{reduced} if every non-leaf element covers at least 2 elements. 
By \cite[Remark 3.3]{defant-kravitz-2023},
every inflated rooted tree poset can be obtained as an inflation of a reduced rooted tree poset, so in the following, we will generally restrict ourselves to reduced rooted tree posets.

Let $\ell_1, \ldots, \ell_m$ denote the leaves of $Q$, where $m$ is the number of leaves. For each $i \in [m]$, we have a unique maximal chain from $\ell_i$ to $r$
\begin{equation}
    \ell_i = u_{i,0} \lessdot_Q u_{i,1} \lessdot_Q \cdots \lessdot_Q u_{i,\omega_i} = r,
\end{equation}
where $\omega_i$ denotes the length of the chain. Recall that $u_{i,0}\lessdot_Q u_{i,1}$ means that $u_{i,1}$ covers $u_{i,0}$ in $Q$. For $i \in [m]$ and $j \in [\omega_i]$, define the two quantities
\begin{equation}
\begin{aligned}
    b_{i,j} &= \sum_{v \le_Q u_{i,j-1}} |\varphi^{-1}(v)|,\\
    c_{i,j} &= \sum_{v <_Q u_{i,j}} |\varphi^{-1}(v)|.
\end{aligned}
\end{equation}
The fraction $\frac{b_{i,j}}{c_{i,j}}$ therefore represents the fraction of elements in $P$ below the minimal element of $\varphi^{-1}(u_{i,j})$ that lie on the preimage of the maximal chain from $\ell_i$ to $r$. When it is necessary to specify the rooted tree poset $Q$, we shall do so by indicating $Q$ in parentheses. For example, we will write $u_{i,j}(Q)$ instead of $u_{i,j}$ or $\omega_i(Q)$ instead of $\omega_i$.

\begin{example}
    The vertices of $Q$ in \cref{fig:InflatedRooted} are labeled in accordance with our definitions above. For example, the maximal chain from $\ell_1$ to $r$ is $\ell_1 = u_{1,0} \lessdot u_{1,1} \lessdot u_{1,2} = r$. The length of this maximal chain is $\omega_1 = 2$. As another example, the maximal chain from $\ell_2$ to $r$ is $\ell_2 = u_{2,0} \lessdot u_{2,1} \lessdot u_{2,2} = r$. Notice that $u_{i,j}$ may refer to the same element in $Q$ for distinct $i$ and $j$. For example, $u_{2,1} = u_{1,1}$ in \cref{fig:InflatedRooted}, and the root $r$ is equal to $u_{1,2}$, $u_{2,2}$, $u_{3,1}$, and $u_{4,1}$.
    
    The quantity $b_{1,1}$ can be computed by 
    $$b_{1,1} = \sum_{v \le_Q u_{1,0}} |\varphi^{-1}(v)|
        = |\varphi^{-1}(u_{1,0})|
        = 4.$$
    Similarly, the quantity $c_{1,1}$ can be computed by $$c_{1,1} = \sum_{v <_Q u_{1,1}} |\varphi^{-1}(v)|
        = |\varphi^{-1}(u_{1,0})| + |\varphi^{-1}(u_{2,0})|
        = 6.$$
    Therefore $\frac{b_{1,1}}{c_{1,1}} = \frac{4}{6}$ of the elements in $P$ below the minimal element of $\varphi^{-1}(u_{1,1})$ lie in the direction of $\varphi^{-1}(\ell_1)$. 
\end{example}

The following technical lemma provides a useful bound for the formula in \cref{theorem:x-labeling-inflated-rooted-forest-formula}. The left  side of \cref{eq:inflation-bound} appears in~\cite[Theorem 3.5]{defant-kravitz-2023}, and a similar term also appears in~\cite[Theorem 9]{hodges-2022}.

\begin{lem}
    \label{lemma:inflation-inequality-lemma}
   Let $Q$ be a reduced rooted tree poset with $m$ leaves and let $P$ be an inflation of $Q$ with $n$ elements. Then
   \begin{equation}
      \label{eq:inflation-bound}
      \sum_{i=1}^m \prod_{j=1}^{\omega_i(Q)} \frac{b_{i,j}(Q) - 1}{c_{i,j}(Q)-1} \le 
      \begin{cases}
      1 & \text{if $n = 1$,}\\
      \frac{n-m}{n-1} & \text{otherwise.}
      \end{cases}
   \end{equation}
\end{lem}
\begin{proof}
    We will prove the bound by inducting on $h(Q) = \max \{\omega_1(Q), \ldots, \omega_m(Q)\}$. The base case is when $h(Q) = 0$. In this case, there is a single leaf so $m = 1$ and $\omega_1(Q) = 0$. Thus, the left side of the inequality is the sum of a single empty product which is equal to $1$. The right side is $1$ regardless of whether $n = 1$ or $n > 1$, so the inequality holds when $h(Q) = 0$.

    Now, suppose $h(Q) > 0$ and that the lemma holds for all rooted tree posets $Q'$ with $h(Q') < h(Q)$. Since $h(Q) > 0$, $n > 1$. Now, let $r$ denote the root of $Q$ and let $q_1, \ldots, q_t$ be the elements covered by $r$. Recall that for an element $x$ in a poset, $\downarrow x$ denotes the set of elements less than or equal to $x$. The subposets $Q_k =\ \downarrow q_k$ are all rooted tree posets with $h(Q_k) \le h(Q) - 1$, and $P_k = \varphi^{-1}(Q_k)$ is an inflation of $Q_k$. Let $n_k = |\varphi^{-1}(Q_k)|$ so that $n - |\varphi^{-1}(r)| = n_1 + \cdots + n_t$, and let $m_k$ denote the number of leaves of $Q_k$ so that $m = m_1 + \cdots + m_t$. For convenience, let $M_k$ denote the $k$th partial sum $m_1 + \cdots + m_k$ and let $M_0 = 0$. Without loss of generality, order the leaves $\ell_1, \ldots, \ell_m$ of $Q$ such that the leaves of $Q_k$ are $\ell_{M_{k-1}+1}, \ldots, \ell_{M_k}$.

    Observe that for $M_{k-1} + 1 \le i \le M_k$, $\omega_i(Q_k) = \omega_i(Q) - 1$, and for $1 \le j \le \omega_i(Q_k)$, $b_{i,j}(Q_k) = b_{i,j}(Q)$ and $c_{i,j}(Q_k) = c_{i,j}(Q)$. Additionally, $b_{i, \omega_i(Q)}(Q) = n_k$ and $c_{i,\omega_i(Q)}(Q) = n - |\varphi^{-1}(r)|$. Thus,
    \begin{align*}
        \sum_{i=1}^m \prod_{j=1}^{\omega_i(Q)} \frac{b_{i,j}(Q) - 1}{c_{i,j}(Q) - 1} &= \sum_{k=1}^t \left(\sum_{i=M_{k-1}+1}^{M_k} \prod_{j=1}^{\omega_i(Q)} \frac{b_{i,j}(Q) - 1}{c_{i,j}(Q)-1}\right)\\
        &= \sum_{k=1}^t \left(\sum_{i=M_{k-1}+1}^{M_k} \frac{n_k-1}{n-|\varphi^{-1}(r)|-1} \cdot \prod_{j=1}^{\omega_i(Q)-1} \frac{b_{i,j}(Q) - 1}{c_{i,j}(Q)-1}\right)\\
        &= \sum_{k=1}^t \frac{n_k-1}{n-|\varphi^{-1}(r)|-1} \cdot \left(\sum_{i=M_{k-1}+1}^{M_k}  \prod_{j=1}^{\omega_i(Q_k)} \frac{b_{i,j}(Q_k) - 1}{c_{i,j}(Q_k)-1}\right).
    \end{align*}
    For each $1 \le k \le t$, if $n_k = m_k = 1$, then we clearly have
    \begin{align*}
        \frac{n_k-1}{n-|\varphi^{-1}(r)|-1} \cdot \left(\sum_{i=M_{k-1}+1}^{M_k} \prod_{j=1}^{\omega_i(Q_k)} \frac{b_{i,j}(Q_k) - 1}{c_{i,j}(Q_k)-1}\right) &\le \frac{n_k-m_k}{n-|\varphi^{-1}(r)|-1},
    \end{align*}
    as both sides of the inequality are 0. If $n_k > 1$, then by the inductive hypothesis we also have
    \begin{align*}
    \frac{n_k-1}{n-|\varphi^{-1}(r)|-1} \cdot \left(\sum_{i=M_{k-1}+1}^{M_k} \prod_{j=1}^{\omega_i(Q_k)} \frac{b_{i,j}(Q_k) - 1}{c_{i,j}(Q_k)-1}\right) &\le \frac{n_k-1}{n-|\varphi^{-1}(r)|-1} \cdot \frac{n_k-m_k}{n_k-1}\\
    &= \frac{n_k-m_k}{n-|\varphi^{-1}(r)|-1}.
    \end{align*}
    Thus, we conclude that
    \begin{equation}
        \label{eq:shoelace-poset-inequality}
        \begin{split}
        \sum_{i=1}^m \prod_{j=1}^{\omega_i(Q)} \frac{b_{i,j}(Q) - 1}{c_{i,j}(Q) - 1} &\le \sum_{k=1}^t \frac{n_k-m_k}{n-|\varphi^{-1}(r)|-1}\\
        &= \frac{n-|\varphi^{-1}(r)|-m}{n-|\varphi^{-1}(r)|-1}\\
        &\le \frac{n-m}{n-1}. 
        \end{split}
    \end{equation}
\end{proof}

\begin{remark}
\label{remark:irfp-strict-ineq}
    Since $|\varphi^{-1}(r)| > 0$, the final inequality in \cref{eq:shoelace-poset-inequality} is strict for $m > 1$. If $m = 1$, then there is only one leaf in $Q$, so $b_{1,j}(Q) = c_{1,j}(Q)$ for $1 \le j \le \omega_1(Q)$ and equality holds. In particular, the upper bound in \cref{lemma:inflation-inequality-lemma} is never sharp for $m > 1$.
\end{remark}

\begin{defn}
    Let $P$ be an $n$ element poset and $X \subseteq P$. A \definition{partial labeling} of $P$ is an injective map $M: X \to [n]$. A labeling $L:P \to [n]$ is an \definition{extension} of $M$ if $L|_X = M$. The set of extensions of $M$ is denoted $\Lambda(P,M)$.
\end{defn}

\begin{defn}
    Let $P$ be a poset and $x \in P$. The element $x$ is \definition{lower order ideal complete (LOI-complete)} if any element that is comparable to some element in $\downarrow x$ is also comparable to $x$ itself.

\end{defn}

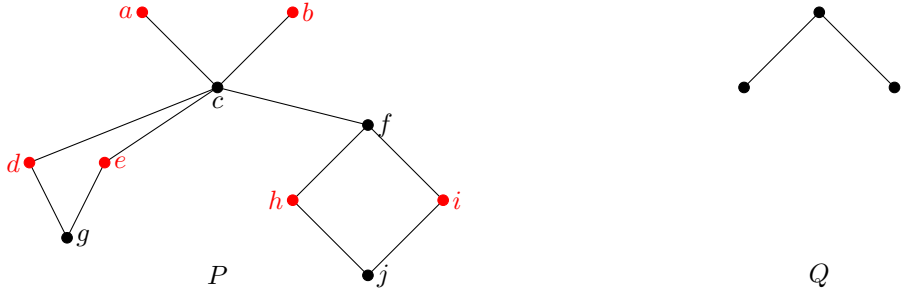
\begin{figure}[ht]
\centering
\begin{tikzpicture}
            \node at (0,-2.5){$P$};
            \draw (0,0) -- (1,1);
            \draw (0,0) -- (-1,1);
            \filldraw[black] (0,0) circle (2pt) node [anchor=north] {$c$};
            \filldraw[red] (1,1) circle (2pt) node [anchor=west]{$b$};
            \filldraw[red] (-1,1) circle (2pt) node [anchor=east] {$a$};
            \draw (0,0) -- (-1.5,-1) -- (-2,-2) -- (-2.5, -1) -- (0,0); 
            \filldraw[red] (-1.5,-1) circle (2pt) node [anchor = west]{$e$};
            \filldraw[red] (-2.5,-1) circle (2pt) node [anchor = east]{$d$};
            \filldraw[black] (-2,-2) circle (2pt) node [anchor = west]{$g$};
            \draw (0,0)-- (2,-.5) -- (3,-1.5) -- (2,-2.5) -- (1,-1.5) -- (2,-.5);
            \filldraw[black] (2,-.5) circle (2pt) node [anchor = west]{$f$};
            \filldraw[red] (3,-1.5) circle (2pt) node [anchor = west]{$i$};
            \filldraw[black] (2,-2.5) circle (2pt) node [anchor = west]{$j$};
            \filldraw[red] (1,-1.5) circle (2pt) node [anchor = east]{$h$};
            \draw (7,0) --(8,1)--(9,0);
            \filldraw[black] (7,0) circle (2pt);
            \filldraw[black] (8,1) circle (2pt);
            \filldraw[black] (9,0) circle (2pt);
            \node at (8,-2.5){$Q$};
        \end{tikzpicture}
\caption{A rooted tree poset $Q$ and an inflation $P$ of $Q$. The LOI-complete elements in $P$ are colored black.}

\label{fig:loi-complete}
\end{figure}
\begin{example}
    \label{example:loi-complete}
    Consider the rooted tree poset $Q$ and its inflation $P$ in \cref{fig:loi-complete}. In $P$, the elements $c$, $f$, $g$, and $j$ are all LOI-complete, since for each of those elements, all elements comparable to $\downarrow c$, $\downarrow f$, $\downarrow g$, and $\downarrow j$ are also comparable to $c$, $f$, $g$, and $j$, respectively. The elements $a$, $b$, $d$, $e$, $h$, and $i$, colored in red, are \emph{not} LOI-complete. For example, $b$ is not LOI-complete because the element $a$ is comparable to $c \in \downarrow b$ but $a$ is not comparable to $b$.
\end{example}

\begin{lem}
    \label{lemma:minimum-element-is-loi-complete}
    Let $Q$ be a rooted tree poset and let $P$ be an inflation of $Q$ with inflation map $\varphi: P \to Q$. For any $q \in Q$, the unique minimal element of $\varphi^{-1}(q)$ is LOI-complete in $P$.
\end{lem}
\begin{proof}
    Denote the unique minimal element of $\varphi^{-1}(q)$ by $x$. Let $y \in \downarrow \varphi^{-1}(q)$ and suppose $z \in P$ is comparable to $y$. If $y = x$ then $z$ and $x$ are comparable by definition. Otherwise, if $y \neq x$, then $\varphi(y) <_Q \varphi(x) = q$ since $x$ is the unique minimal element of $\varphi^{-1}(q)$. Since $z \in P$ is comparable to $y$ and $P$ is a rooted tree poset, $\varphi(z)$ and $\varphi(y)$ are comparable and hence $\varphi(z)$ and $q$ is comparable.
    
    If $\varphi(z) <_Q q$, then $z <_P x$ by \cref{defn:inflation}. If $\varphi(z) = q$, then $x \le_P z$ since $x$ is the unique minimal element of $\varphi^{-1}(q)$. If $q <_Q \varphi(z)$, then $x <_P z$. In each case, $z$ is comparable to $x$. Therefore $x$ is LOI-complete in $P$.
\end{proof}

We will need the following probability lemmas from \cite{defant-kravitz-2023}, so we have reproduced them for convenience.

\begin{lem}[{\cite[Lemma~3.10]{defant-kravitz-2023}}]
    \label{lemma:probability-other}
    Let $P$ be an $n$ element poset and $x \in P$ be LOI-complete. Let $X =~ \downarrow x \setminus \{x\}$. For $L \in \Lambda(P)$ and $k \ge 0$, the set $L_k(X)$ depends only on the set $L(X)$ and the restriction $L|_{P \setminus X}$. It does not depend on the way in which labels in $L(X)$ are distributed among the elements of $X$.
\end{lem}

\begin{lem}[{\cite[Lemma~3.11]{defant-kravitz-2023}}] 
    \label{lemma:probability}
    Let $P$ be an $n$ element poset and $x \in P$ be LOI-complete. Let $X =~ \downarrow x \setminus \{x\}$ and suppose that $X \neq \varnothing$. Let $A \subseteq X$  have the property that no element of $A$ is comparable with any element in $X \setminus A$ and let $M: P \setminus X \to [n]$ be a partial labeling such that $L_{n-1}^{-1}(1) \in X$ for every extension $L$ of $M$. If a labeling $L$ is chosen uniformly at random from the extensions in $\Lambda(P, M)$, then the probability that $L^{-1}_{n-1}(1) \in A$ is $\frac{|A|}{|X|}$.
\end{lem}

By suitably modifying the proof of \cite[Theorem 3.5]{defant-kravitz-2023}, one can strengthen it to obtain \cref{theorem:x-labeling-inflated-rooted-forest-formula}. The following proof is self-contained, but the interested reader may wish to refer to \cite[Section 3]{defant-kravitz-2023} for further details.

\begin{thm}
\label{theorem:x-labeling-inflated-rooted-forest-formula}
    Let $Q$ be a reduced rooted tree poset with $m$ leaves and let $P$ be an inflation of $Q$ with $n $ elements, with inflation map $\varphi: P \to Q$. For a nonminimal element $x \in P$, let $\ell(x) = \{i \in [m] : \ell_i \le_Q \varphi(x)\}$ and $\omega_{i,x} = \max \{j : u_{i,j} \le_Q \varphi(x)\}$. Then the number of tangled $x$-labelings of $P$ is given by
    \begin{equation*}
        |\T_x(P)| = (n-2)! \sum_{i \in \ell(x)} \prod_{j = 1}^{\omega_{i,x}} \frac{b_{i,j}(Q)-1}{c_{i,j}(Q)-1}.
    \end{equation*}
\end{thm}
\begin{proof}
     Fix a leaf $\ell_i$ of $Q$ and let $x_0$ be the unique minimal element of $\varphi^{-1}(\ell_i)$. We will count the number of tangled labelings $L$ such that $L^{-1}(n) = x_0$ and $L^{-1}(n-1) = x$. By \cref{cor:(n-r)-inequality-when-tangled}, if $L$ is tangled, then $x_0 = L^{-1}(n) <_P L^{-1}(n-1) = x$. Thus, we need only consider leaves $\ell_i$ such that $\ell_i \le_Q \varphi(x)$. Furthermore, since $Q$ is reduced, $L$ is tangled if and only if $L^{-1}_{n-2}(1) \in \varphi^{-1}(\ell_i)$.

     If $\omega_{i,x} = 0$, then the product $\prod_{j=1}^{\omega_{i,x}} \frac{b_{i,j}(Q) -1}{c_{i,j}(Q)-1}$ is the empty product 1. In this case, $x \in \varphi^{-1}(\ell_i)$ so all $x$-labelings $L$ such that $L^{-1}(n) = x_0$ are tangled.

     Now, assume $\omega_{i,x} \ge 1$ and choose a labeling $L \in \Lambda(P)$ uniformly at random among the $(n-2)!$ labelings that satisfy $L^{-1}(n) = x_0$ and $L^{-1}(n-1) = x$. We will proceed to compute the probability that $L$ is tangled. Let $\widetilde{P} = P \setminus \{x_0\}$ and $\widetilde{\varphi} = \varphi|_{\widetilde{P}}$. For $1 \le j \le \omega_{i,x}$, let $x_j$ be the unique minimal element of $\widetilde{\varphi}^{-1}(u_{i,j})$, and define the sets
     \[
     X_j = \downarrow x_j \setminus \{x_j\} \text{ and } A_j = \bigcup_{v \le_Q u_{i,j-1}} \widetilde{\varphi}^{-1}(v).
     \]
     The sizes of the sets are $|X_j| = b_{i,j}(Q) - 1$ and $|A_j| = c_{i,j}(Q) - 1$.

     For any partial labeling $M: \widetilde{P} \setminus X_{\omega_{i,x}} \to [n-1]$ such that $M(x) = n-1$ and any extension $L$ of $M$, the condition $L^{-1}_{n-2}(1) \in X_{\omega_{i,x}}$ holds since $x \in \varphi^{-1}(u_{i,\omega_{i,x}})$. Furthermore, since $P$ is an inflated rooted forest poset and $x_j$ is the unique minimal element of $\widetilde{\varphi}^{-1}(u_{i,j})$, $x_j$ is LOI-complete, and no element of $A_j$ is comparable with any element of $X_j \setminus A_j$. Thus, the poset $\widetilde{P}$, the subsets $X_{\omega_{i,x}}$ and $A_{\omega_{i,x}}$, and the partial labeling $M$ satisfy the conditions in \cref{lemma:probability}. Applying the lemma tells us that the probability that $L_{n-2}^{-1}(1) \in A_{\omega_{i,x}}$ is
     \[
        \frac{|A_{\omega_{i,x}}|}{|X_{\omega_{i,x}}|} = \frac{b_{i,\omega_{i,x}}(Q)-1}{c_{i,\omega_{i,x}}(Q)-1}.
     \]
     Furthermore, \cref{lemma:probability-other} tells us that the occurrence of this event only depends on $L|_{\widetilde{P}\setminus A_{\omega_{i,x}}}$.

     This process can be continued for $j = \omega_{i,x}-1, \ldots, 1$ to deduce that the probability that $L_{n-2}^{-1}(1) \in \varphi^{-1}(u_{i,0})$ is the product
     \[
     \prod_{j=1}^{\omega_{i,x}} \frac{b_{i,j}(Q)-1}{c_{i,j}(Q)-1}.
     \]
     Summing over all the leaves such that $\ell_i \le_Q \varphi(x)$ yields the result.
\end{proof}

\begin{thm}\label{theorem:inflated_rooted}
    If $P$ is an inflated rooted forest poset on $n$ elements and $x\in P$, then $|\T_x(P)| \le (n-2)!$. Equality holds if and only if there is a unique minimal element $z \in P$ such that $z <_P x$.
\end{thm}
\begin{proof}
    We first consider the case of an inflated rooted tree poset. Let $Q$ be a reduced rooted tree poset and $P$ an inflation of $Q$ with $|P| = n$. For an element $x$ of $P$, \cref{theorem:x-labeling-inflated-rooted-forest-formula} implies that
    \[
    |\T_x(P)| = (n-2)!\sum_{i \in \ell(x)}\prod_{j=1}^{\omega_{i,x}} \frac{b_{i,j}(Q) - 1}{c_{i,j}(Q)-1}.
    \]
    The subposet $\widetilde{Q} := \downarrow\varphi(x)$ is also a rooted tree poset. Let $\widetilde{P} := \varphi^{-1}(\widetilde{Q})$ and $\widetilde{\varphi}$ be the restriction $\varphi|_{\widetilde{P}}$. Then $\widetilde{P}$ is an inflated rooted tree poset, so \cref{lemma:inflation-inequality-lemma} gives the upper bound
    \begin{equation}
    \label{equation:inflated_rooted_bound}
    \sum_{i \in \ell(x)}\prod_{j=1}^{\omega_{i,x}} \frac{b_{i,j}(Q)-1}{c_{i,j}(Q)-1} \le 1.
    \end{equation}
    Therefore, $|\T_x(P)| \le (n-2)!$ in the case of an inflated rooted tree poset.
    
    Let $m$ denote the number of leaves in the subposet $\widetilde{Q}$. By \cref{remark:irfp-strict-ineq}, the inequality in \cref{equation:inflated_rooted_bound} is strict if and only if $m > 1$. The number of leaves in the subposet $\widetilde{Q}$ is precisely the number of minimal elements in $\widetilde{P}$. By definition of $\widetilde{P}$, the minimal elements in $\widetilde{P}$ are precisely the minimal elements $z \in P$ that satisfy $z <_P x$. Thus, equality in \cref{equation:inflated_rooted_bound} holds if and only if there is a unique minimal element $z \in P$ that satisfies $z <_P x$.
    
    The general case of an inflated rooted forest poset follows from \cref{lemma:n-2-conjecture-disjoint}, since an inflated rooted forest poset is a disjoint union of inflated rooted tree posets. 
\end{proof}

\section{Shoelace Posets}
\label{section:shoelace-posets}

In this section, we will study tangled labelings on a new family of posets called \emph{shoelace posets} and show that 
the $(n-2)!$ conjecture holds for them. The key ingredient in the proof is a careful analysis of the number of tangled labelings where a fixed element in the poset is labeled $n-1$. We note that in general, shoelace posets are not the inflation of any rooted forest poset. We will also examine a specific subset of shoelace posets called \emph{$W$-posets}, and enumerate the exact number of tangled labelings of these posets.

\begin{defn}
    A \definition{shoelace poset} $P$ is a connected poset defined by a set of minimal elements $\{x_1, \ldots, x_\ell\}$, a set of maximal elements $\{y_1, \ldots, y_m\}$, and a set $\mathcal{S}(P) \subseteq \{x_1, \ldots, x_\ell\} \times \{y_1, \ldots, y_m\}$ such that the following three conditions hold:
    \begin{enumerate}
        \item For every $(i,j) \in [\ell] \times [m]$, the elements $x_i$ and $y_j$ are comparable in $P$ if and only if $(x_i, y_j) \in \mathcal{S}(P)$.
        \item For every $(x_i, y_j) \in \mathcal{S}(P)$, the open interval $(x_i, y_j)_P$ is a (possibly empty) chain, denoted $C_i^j$.
        \item For distinct pairs $(x_i,y_j),(x_{i'},y_{j'}) \in \mathcal{S}(P)$, the chains $C_i^j$ and $C_{i'}^{j'}$ are disjoint.
     \end{enumerate}
\end{defn}

We will use the following notation
\[
\mathcal{S}^j(P)=\{x_i : (x_i,y_j) \in \mathcal{S}(P)\},\qquad \mathcal{S}_i(P)=\{y_j : (x_i,y_j) \in \mathcal{S}(P)\}.
\]

The funnels of a shoelace poset can be described fairly simply. The funnel of a minimal element $x_i$ consists of the elements in $C_i^j$ for $y_j \in \mathcal{S}_i(P)$, along with the maximal elements $y_j$ for $y_j \in \mathcal{S}_i(P)$ that satisfy $\mathcal{S}^j(P) = \{x_i\}$.

\begin{example}
\cref{fig:shoelace} depicts a shoelace poset $P$ with 3 minimal elements and 4 maximal elements. In this example $\mathcal{S}(P) = \{(x_1,y_2), (x_1,y_3), (x_1,y_4), (x_2,y_1), (x_2,y_3), (x_3,y_3), (x_3,y_4)\}$. The elements of the chain $C_3^4$ are highlighted in blue and the chain $C_1^3$ is empty. Notice also that $\mathcal{S}_1(P) = \{y_2,y_3,y_4\}$ and $\mathcal{S}^2(P) = \{x_1\}$.
\end{example}
\begin{figure}[!ht]
    \centering
        \begin{tikzpicture}
            \draw (-3,2) -- (-1,0) -- (0,2) -- (1,0) -- (3,2) -- (-3,0);
            \draw (-2,2) -- (-3,0) -- (0,2);
            \filldraw (-3,2) circle (2pt) node [anchor=west]{$y_1$};
            \filldraw (-1,0) circle (2pt) node [anchor=north west]{$x_2$};
            \filldraw (0,2) circle (2pt) node [anchor=west]{$y_3$};
            \filldraw (1,0) circle (2pt) node [anchor=north west]{$x_3$};
            \filldraw (3,2) circle (2pt) node [anchor=west]{$y_4$};
            \filldraw (-3,0) circle (2pt) node [anchor=north west]{$x_1$};
            \filldraw (-2,2) circle (2pt) node [anchor=west]{$y_2$};
            \filldraw[blue] (2,1) circle (1.5pt) node [anchor=north west]{$C_3^4$};
            \filldraw[blue] (2.5,1.5) circle (1.5pt) node [anchor=north]{};
            \filldraw[blue] (1.5,.5) circle (1.5pt) node [anchor=west]{};
            \filldraw (-2,1) circle (1pt) node [anchor=west]{};
            \filldraw (-2.23,1.53) circle (1pt) node [anchor=west]{};
            \filldraw (-2.66,.66) circle (1pt) node [anchor=west]{};
            \filldraw (0,1) circle (1pt) node [anchor=west]{};
            \filldraw (-.69,.6) circle (1pt) node [anchor=west]{};
            \filldraw (-.33,1.33) circle (1pt) node [anchor=west]{};
            \filldraw (-2,.33) circle (1pt) node [anchor=west]{};
            \filldraw (-1,.66) circle (1pt) node [anchor=west]{};
            \filldraw (1,1.33) circle (1pt) node [anchor=west]{};
            \filldraw (2,1.66) circle (1pt) node [anchor=west]{};
        \end{tikzpicture}

    \caption{An example of a shoelace poset.}
    \label{fig:shoelace}
\end{figure}
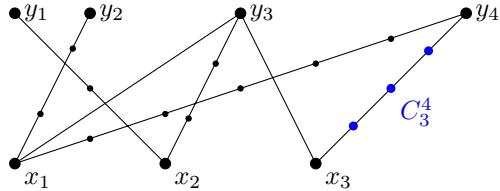

In order to prove that shoelace posets satisfy the $(n-2)!$ conjecture, we will partition labelings according to the location of the label $n-1$, and bound $|\T_x(P)|$ for the various elements $x$.

For the following lemma, we use the following notation: for $S$ a set and $f$ a function whose codomain is well-ordered, $\argmin_{S} f$ is the element $x \in S$ such that $f(x)$ is minimal. 
\begin{lem}
    \label{lemma:shoelace-necessary-tangled-condition}
    Let $P$ be a shoelace poset with minimal elements $x_1, \ldots, x_\ell$ and maximal elements $y_1, \ldots, y_m$. Let $L \in \T(P)$, $i \in [\ell]$, and $j \in [m]$ such that $L(y_j) = n-1$ and $L(x_i) = n$. If $|\mathcal{S}^j(P)| \ge 2$, then $x_i \in \mathcal{S}^j(P)$, $C_i^j \neq \varnothing$, and 
    \begin{equation*}
        \underset{\downarrow y_j \setminus \mathcal{S}^j(P)}{\argmin} L \in C_i^j.
    \end{equation*}
\end{lem}

\begin{proof}
    Since $L$ is tangled, $L^{-1}(n) <_P L^{-1}(n-1)$ by \cref{cor:(n-r)-inequality-when-tangled}. Therefore, $x_i <_P y_j$, which implies $i \in \mathcal{S}^j(P)$. By \cref{lemma:tangled-funnel-condition}, there exists $z \in \funnel(x_i)$ such that $z \le_P y_j$. By the assumption that $|\mathcal{S}^j(P)| \ge 2$, we observe that $y_j \not\in \funnel(x_i)$. Therefore, $x_i <_P z <_P y_j$, so $C_i^j \neq \varnothing$.

    Next, let $r$ be the smallest positive integer such that the $r$th promotion chain ends in $y_j$. Denote the $r$th promotion chain by $(z_1, \ldots, z_n, y_j)$. Since $r$ is the smallest such positive integer, $y_j$ does not lie on the $q$th promotion chain for $q < r$, and hence $L_{r-1}^{-1}(n-1-(r-1)) = y_j$. Then, after the $r$th promotion, $L_r^{-1}(n-1-r) = z_n$. Since $L$ is a tangled labeling, \cref{cor:(n-r)-inequality-when-tangled} implies that 
    \begin{equation*}
        x_i = L_r^{-1}(n-r) <_P L_r^{-1}(n-1-r) = z_n.
    \end{equation*}
    Therefore, $z_n \in C_i^j$. Since $z_1 <_P \ldots <_P z_{n-1}$, the remaining elements $z_1, \ldots, z_{n-1}$ in the $r$th promotion chain are also on $C_i^j$. 

    Now, let $z \in\, \downarrow y_j \setminus \mathcal{S}^j(P)$ and let $t = L(z)$. Then either $L_{t-1}^{-1}(1) = z$ and the $t$th promotion chain ends in $y_j$, or $L_{t-1}^{-1}(1) <_P z$ and the $t'$-th promotion chain ends in $y_j$ for some $t' < t$. In either case, it follows that $r \le t$. Since the starting element of the $r$th promotion chain lies in $C_i^j$, we conclude that $\underset{\downarrow y_j \setminus \mathcal{S}^j(P)}{\argmin} L \in C_i^j$.
\end{proof}
    
Essentially, if a labeling on a shoelace poset is tangled, and $L(y_j) = n-1$, then the element with smallest label in $\downarrow y_j \setminus S^j(P)$ must be above the element labeled $n$. This is therefore a necessary condition for a labeling on a shoelace poset to be tangled. This will be instrumental in proving the following theorem.
  
\begin{thm}\label{thm:shoelace}
    If $P$ is a shoelace poset on $n$ elements and $z\in P$, then $|\T_z(P)| \le (n-2)!$. Equality holds if and only if there is a unique minimal element $x <_P z$.
\end{thm}
\begin{proof}
Let $x_1, \ldots, x_\ell$ be minimal elements of $P$, and $y_1, \ldots, y_m$ be maximal elements of $P$. The element $z$ can either be a minimal element, an element on a chain $C_i^j$ for some $i$ and $j$, or a maximal element. There is a unique minimal element $x <_P z$ only if $z\in C_i^j$ or if $z$ is one of the maximal elements $y_j$ and $|\mathcal{S}^j(P)| = 1$. For convenience, we set $s := |\mathcal{S}^j(P)|$. Below we separate the cases mentioned above and claim that equality holds only in Case 2 and Case 3.

\textbf{Case 1: Suppose $z$ is a minimal element.} In this case, it is impossible to find an element labeled $n$ such that $L^{-1}(n) <_P L^{-1}(n-1)=z$. So by \cref{cor:(n-r)-inequality-when-tangled}, $|\T_z(P)| = 0$.

\textbf{Case 2: Suppose $z$ lies on some chain $C_i^j$.} In this case there is a unique basin $x_i$ that in $\downarrow z$. Any tangled labeling $L \in \T_z(P)$ must satisfy $L(z) = n-1$ and $L(x_i) = n$. There are at most $(n-2)!$ such labelings, and by \cref{lemma:sufficient-condition-tangled} all such labelings are tangled so $|\T_z(P)| = (n-2)!$. 

\textbf{Case 3: Suppose $z$ is a maximal element $y_j$ and $s = 1$.} Since $s=1$, for any tangled labeling $L$, $L^{-1}(n)$ must be the unique $x_i$ satisfying $x_i<_P y_j=z$. There are $(n-2)!$ such labelings, and by \cref{lemma:sufficient-condition-tangled} all such labelings are tangled. Thus, $|\T_z(P)| = (n-2)!$. 

\textbf{Case 4: Suppose $z$ is a maximal element $y_j$ and $s \ge 2$.} Partition $\Lambda(P)$ into equivalence classes, where two labelings $L$ and $L'$ belong to the same equivalence class if and only if they restrict to the same labeling on $P \setminus \mathcal{S}^j(P)$. Labelings in $\T_z(P)$ require $y_j$ to be labeled $n-1$ and some element in $\mathcal{S}^j(P)$ to be labeled $n$. The number of equivalence classes where this is possible is $(n-2)(n-3) \cdots s$. In each such equivalence class, the tangled labelings $L$ have only one choice of $L^{-1}(n)$ according to \cref{lemma:shoelace-necessary-tangled-condition}. Therefore, at most $(s-1)!$ labelings in each equivalence class are tangled. Consequently, $|\T_z(P)| \le (n-2)(n-3) \cdots s(s-1)! = (n-2)!$.

With a little more careful analysis, one can conclude that at least one of the equivalence classes has strictly fewer than $(s-1)!$ labelings. Consider an equivalence class where the label 1 is in $\mathcal{S}^j(P)$ and the label 2 is in $\downarrow y_j \setminus \mathcal{S}^j(P)$. Then in this equivalence class, there is the additional restriction $L^{-1}(1) \not<_P L^{-1}(2)$. Thus, there are strictly fewer than $(s-1)!$ tangled labelings, so $|\T_z(P)| < (n-2)!$.
\end{proof}

Notice that \cref{thm:shoelace} shows that shoelaces satisfy \cref{conj:(n-1)-refinement}, and therefore also satisfy \cref{conjecture:hodges} and \cref{conj:main-conjecture}.

We have proven an upper bound on the number of tangled labelings of shoelaces, but we are also able to enumerate the exact number of tangled labelings for a specific subfamily of shoelace posets called \emph{$W$-posets}. In general, few explicit formulas for tangled labelings are known. The proof of this formula will also involve counting the number of tangled labelings by fixing the label $n-1$.

\begin{defn}
    Given $a,b,c,d \in \mathbb{Z}_{\ge 0}$, the \definition{$W$-poset} $W_{a,b,c,d}$ is a poset on $a+b+c+d+3$ elements: $\alpha_1, \ldots, \alpha_a$, $\beta_1, \ldots, \beta_b$, $\gamma_1, \ldots, \gamma_c$, $\delta_1, \ldots, \delta_d$, $x,y,z$. The partial order has covering relations $\alpha_i \lessdot_W \alpha_{i+1}$, $\beta_i \lessdot_W \beta_{i+1}$, $\gamma_i \lessdot_W \gamma_{i+1}$, $\delta_i \lessdot_W \delta_{i+1}$, $x \lessdot_W \alpha_1$, $x \lessdot_W \beta_1$, $\beta_b \lessdot_W y$, $\gamma_c \lessdot_W y$, $z \lessdot_W \gamma_1$, and $z \lessdot_W \delta_1$. 
\end{defn}

The poset $W_{a,b,c,d}$ can be viewed as the shoelace poset with the set of minimal elements $\{x,z\}$, the set of maximal elements $\{\alpha_a, y, \delta_d\}$ and $\mathcal{S}(P) = \{(x,\alpha_a),(x,y),(z,y),(z,\delta_d)\}$.

\begin{example}
The Hasse diagram for $W_{2,2,1,1}$ is shown in \cref{figure:W_2211}. There are 34,412 tangled labelings of this poset. 

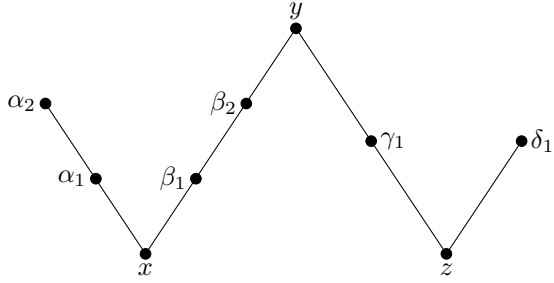
\begin{figure}[!h]
        \centering
        \begin{tikzpicture}
            \draw (-3.33,2) -- (-2,0) --  (0,3) -- (2,0) -- (3,1.5);
            \filldraw[black] (-3.33,2) circle (2pt) node [anchor=east] {$\alpha_2$};
            \filldraw[black] (-2.66,1) circle (2pt) node [anchor=east] {$\alpha_1$};
            \filldraw[black] (-2,0) circle (2pt) node [anchor=north] {$x$};
            \filldraw[black] (-1.33,1) circle (2pt) node [anchor=east] {$\beta_1$};
            \filldraw[black] (-.66,2) circle (2pt) node [anchor=east] {$\beta_2$};
            \filldraw[black] (0,3) circle (2pt) node [anchor=south] {$y$};
            \filldraw[black] (1,1.5) circle (2pt) node [anchor=west] {$\gamma_1$};
            \filldraw[black] (2,0) circle (2pt) node [anchor=north] {$z$};
            \filldraw[black] (3,1.5) circle (2pt) node [anchor=west] {$\delta_1$};
        \end{tikzpicture}
        \caption{The poset $W_{2,2,1,1}$.}
        \label{figure:W_2211}
    \end{figure}
\end{example}

\begin{thm}\label{thm:w-tangled}
    Let $a,b,c,d$ be four positive integers and $n = a+b+c+d+3$. Let
    \begin{align*}
        X &= \binom{n-2}{a} \sum_{i=0}^{b-1} \sum_{j=0}^d (d-j+1)\binom{i+j+c-1}{i,j,c-1}\text{, and}\\
        Z &= \binom{n-2}{d} \sum_{i=0}^{c-1} \sum_{j=0}^a (a-j+1)\binom{i+j+b-1}{i,j,b-1}.
    \end{align*}
    
    Then the number of tangled labelings of $W_{a,b,c,d}$ is given by $(n-2)(n-2)! - a!b!c!d!(X+Z)$.
\end{thm}

\begin{proof}
    Fix $a,b,c,d$ and write $W=W_{a,b,c,d}$. By \cref{equation:tangled-refinement}, it suffices to compute $|\T_p(W)|$ as $p$ ranges over elements of $W$. If $p = x$ or $p = z$, then $|\T_p(W)| = 0$ due to Case 1 in the proof of \cref{thm:shoelace}. If $p = \alpha_i$ or $p = \beta_i$, then this belongs to Cases 2 and 3 in the proof of \cref{thm:shoelace}, and so $|\T_p(W)| = (n-2)!$. Similarly, if $p = \gamma_i$ or $p = \delta_i$, then $|\T_p(W)| = (n-2)!$. With the exception of $p = y$, we have counted $(a+b+c+d)(n-2)! = (n-3)(n-2)!$ tangled labelings.    

    Let us now count the number of tangled labelings $L$ that satisfy $L(y) = n-1$. Observe that permuting the labels $L(\alpha_1), \ldots, L(\alpha_a)$ does not change whether or not $L$ is tangled. Similarly, permuting the labels $L(\beta_1), \ldots, L(\beta_b)$, the labels $L(\gamma_1), \ldots, L(\gamma_c)$, and the labels $L(\delta_1), \ldots, L(\delta_d)$ among themselves does not change whether or not $L$ is tangled. Thus, we will additionally impose the conditions $L(\alpha_1) < \cdots < L(\alpha_a)$, $L(\beta_1) < \cdots < L(\beta_b)$, $L(\gamma_1) < \cdots < L(\gamma_c)$, and $L(\delta_1) < \cdots < L(\delta_d)$. To obtain the total number of tangled labelings, we will count the number of such tangled labelings $L$ satisfying these conditions and then multiply by $a!b!c!d!$.

    We split into two cases. The first case is where $L(\beta_1) < L(\gamma_1)$. Let $m_\beta = L(\beta_1)$. In this case, a necessary condition for $L$ to be tangled is that $L(x) = n$. To see this, suppose otherwise that $L(z) = n$. Then note that  $L_{m_\beta}^{-1}(n-1-m_\beta) \in [x, y)$. This is because for the first $m_\beta$ promotions, the only promotion chains ending in $y$ are those that begin with some element in $[x,y)$ and furthermore, there exists at least one promotion chain ending in $y$, namely  the $m_\beta$-th one. It follows that $L_{n-2}^{-1}(1) \not >_W z$ so $L$ cannot be tangled if $L(z) = n$ (\cref{lemma:tangled-labelings-characterization}).
    
    Now, the total number of labelings that satisfy all these conditions is given by $\frac{1}{2}\binom{n-2}{a,b,c,d,1}$, since it amounts to choosing $a$ of the labels in $[n-2]$ for $\alpha_1, \ldots, \alpha_a$, $b$ of the labels for the $\beta$s and so on. To account for the condition $L(\beta_1) < L(\gamma_1)$, we divide by 2 because there is an involution swapping $L(\beta_1)$ and $L(\gamma_1)$. We will now subtract the number of labelings satisfying these conditions that are \emph{not} tangled. 

    Given that $L$ satisfies all the conditions above, $L$ is \emph{not} tangled if and only if $L(z) < L(\beta_1)$ and there do not exist $\delta_i$
    such that $L(\beta_1) < L(\delta_i) < L(\gamma_1)$. To see this, observe that $L$ is not tangled if and only if there is some $j < m_\beta$ where the $j$th promotion chain begins with an element in $[z, y)$ and ends in $y$. Since $L(\beta_1) < L(\gamma_1)$, this can occur only if $L(z) < L(\beta_1)$. Now, let $\delta_{i} <_W \delta_{i+1} <_W \cdots <_W \delta_{j}$ be all the $\delta$'s with labels in between $L(z)$ and $L(\gamma_1)$. Then the $L(z), L(\delta_i), \ldots, L(\delta_{j-1})$th promotion chains would all begin with $z$ and end with some $\delta_k$, and the $L(\delta_j)$th promotion chain would begin with $z$ and end with $y$. Thus, in order for $L$ to not be tangled we must have $L(\delta_j) < L(\beta_1)$. And conversely, if we do have $L(\delta_j) < L(\beta_1)$ then $L$ is not tangled since the $L(\delta_j)$th promotion chain would start with $z$ and end with $y$.

    Now, we wish to count the number of such labelings $L$. To do so, observe that the labels of the $\alpha$'s are subject to no constraints. We will suppose that $L(\delta_1) < \cdots < L(\delta_{d-j}) < L(\beta_1) < \cdots < L(\beta_{b-i}) < L(\gamma_1)$ and sum over $0 \le i \le b-1$ and $0 \le j \le d$.

    For each $i,j$ there are $(d-j+1)$ choices of what $L(z)$ could be and $\binom{i+j+c-1}{i,j, c-1}$ choices for the labels greater than $L(\gamma_1)$.
    This yields
    \begin{equation*}
        X = \binom{n-2}{a} \sum_{i=0}^{b-1} \sum_{j=0}^d (d-j+1)\binom{i+j+c-1}{i,j,c-1}.
    \end{equation*}
    
    By a similar argument, if $L(\gamma_1) < L(\beta_1)$ then a necessary condition for $L$ to be tangled is $L(z) = n$. The number of labelings satisfying these conditions is $\frac{1}{2}\binom{n-2}{a,b,c,d,1}$ and the number of these labelings that are not tangled is
    \begin{equation*}
        Z = \binom{n-2}{d} \sum_{i=0}^{c-1} \sum_{j=0}^a (a-j+1)\binom{i+j+b-1}{i,j,b-1}.
    \end{equation*}
    Thus, the number of tangled labelings $L$ that satisfy $L(y) = n-1$ is
    \begin{align*}
        a!b!c!d!\left(\frac{1}{2}\binom{n-2}{a,b,c,d,1} - X + \frac{1}{2}\binom{n-2}{a,b,c,d,1} - Z\right) &= a!b!c!d!\binom{n-2}{a,b,c,d,1 } - a!b!c!d!(X + Z)\\
        &= (n-2)! - a!b!c!d!(X+Z).
    \end{align*}
    Adding this to the $(n-3)(n-2)!$ tangled labelings where $L^{-1}(n-1) \neq y$ yields the desired formula.
\end{proof}

In principle, one could compute the exact number of tangled labelings for various subsets of shoelace posets in this way. Even for the class of $W$-posets, however, the computations appear rather unwieldy.


\section{Generating Functions}
\label{section:generating-functions}

In the previous sections, we focused on counting the number of tangled labelings of various posets and analyzed their upper bounds. In this section, we are interested in exploring the number of labelings of a poset $P$ on $n$ elements that have a fixed order $k$. Recall that the order of a labeling $L$ is the minimal integer $k \ge 0$ such that $L_k$ is sorted. Such labelings we will call $k$-sorted; see \cref{defn:k-tangled}. Dual to $k$-sorted labelings are $k$-tangled labelings that have order $n-k-1$. We define two kinds of generating functions (\cref{defn:gfs}) on $P$ and investigate how these generating functions change if we attach some minimal elements to $P$. Our result provides a simple and unified proof of enumerating tangled labelings and quasi-tangled labelings in \cite{defant-kravitz-2023} and \cite{hodges-2022} (see \cref{rmk:(quasi)_tangled_p(l)}).

\begin{defn}\label{defn:k-tangled}
    Let $P$ be an $n$-element poset. A labeling $L \in \Lambda(P)$ is said to be \definition{$k$-sorted} if $\order(L) = k$ and is said to be \definition{$k$-tangled} if $\order(L) = n-k-1$.
\end{defn}

Observe that natural labelings are synonymous with $0$-sorted labelings and tangled labelings are synonymous with $0$-tangled labelings. Quasi-tangled labelings introduced in \cite{hodges-2022} correspond exactly to $1$-tangled labelings.

\begin{defn}\label{defn:gfs}
    Let $P$ be an $n$-element poset. The  \definition{sorting generating function} of $P$ is defined to be 
    \begin{equation*}
        f_P(q) \coloneqq \sum_{L \in \Lambda(P)} q^{\order(L)} = \sum_{i=0}^{n-1}a_iq^{i},
    \end{equation*}
    where $a_i$ counts the number of $i$-sorted labelings of $P$. The  \definition{cumulative generating function} of $P$ is defined to be  
    \begin{equation*}
        g_P(q) \coloneqq \sum_{i=0}^{n-1}b_iq^i,
    \end{equation*}
    where $b_i \coloneqq a_0 + a_1 + \cdots + a_i$ is the partial sum of $a_i$'s. In particular, $b_{n-1} = n!$. 
\end{defn}

\begin{example}\label{ex:fg}
We list all the six labelings and their orders of the $\Lambda$-shaped poset $P$ in \cref{tab:posetV}. The sorting generating function and cumulative generating function of $P$ are given by $f_P(q)=2+4q$ and $g_P(q) = 2+6q+6q^2$.

\begin{table}[htb!]
    \centering
    \begin{tabular}{c|c|c|c|c|c|c}
      Labeling   & \begin{tikzpicture}
        \draw (-0.55,0) -- (0,1) --  (0.55,0);
        \filldraw[black] (-0.55,0) circle (2pt) node [anchor=east] {$2$};
        \filldraw[black] (0,1) circle (2pt) node [anchor=south] {$1$};
        \filldraw[black] (0.55,0) circle (2pt) node [anchor=west] {$3$};
        \end{tikzpicture} & \begin{tikzpicture}
        \draw (-0.55,0) -- (0,1) --  (0.55,0);
        \filldraw[black] (-0.55,0) circle (2pt) node [anchor=east] {$3$};
        \filldraw[black] (0,1) circle (2pt) node [anchor=south] {$1$};
        \filldraw[black] (0.55,0) circle (2pt) node [anchor=west] {$2$};
        \end{tikzpicture} & \begin{tikzpicture}
        \draw (-0.55,0) -- (0,1) --  (0.55,0);
        \filldraw[black] (-0.55,0) circle (2pt) node [anchor=east] {$1$};
        \filldraw[black] (0,1) circle (2pt) node [anchor=south] {$2$};
        \filldraw[black] (0.55,0) circle (2pt) node [anchor=west] {$3$};
        \end{tikzpicture} & \begin{tikzpicture}
        \draw (-0.55,0) -- (0,1) --  (0.55,0);
        \filldraw[black] (-0.55,0) circle (2pt) node [anchor=east] {$3$};
        \filldraw[black] (0,1) circle (2pt) node [anchor=south] {$2$};
        \filldraw[black] (0.55,0) circle (2pt) node [anchor=west] {$1$};
        \end{tikzpicture} & \begin{tikzpicture}
        \draw (-0.55,0) -- (0,1) --  (0.55,0);
        \filldraw[black] (-0.55,0) circle (2pt) node [anchor=east] {$1$};
        \filldraw[black] (0,1) circle (2pt) node [anchor=south] {$3$};
        \filldraw[black] (0.55,0) circle (2pt) node [anchor=west] {$2$};
        \end{tikzpicture} & \begin{tikzpicture}
        \draw (-0.55,0) -- (0,1) --  (0.55,0);
        \filldraw[black] (-0.55,0) circle (2pt) node [anchor=east] {$2$};
        \filldraw[black] (0,1) circle (2pt) node [anchor=south] {$3$};
        \filldraw[black] (0.55,0) circle (2pt) node [anchor=west] {$1$};    
        \end{tikzpicture} \\
      \hline
      Order & 1 & 1 & 1 & 1 & 0 & 0 
    \end{tabular}
    \caption{The six labelings and their corresponding orders for the $\Lambda$-shaped poset.}
    \label{tab:posetV}
\end{table}
\end{example}

We now define precisely what it means to attach $k$ minimal elements to a poset. The operation we need is the ordinal sum of two posets $P$ and $Q$.
\begin{defn}
    Let $P$ and $Q$ be two posets. The \definition{ordinal sum} of $P$ and $Q$ is the poset $P \oplus Q$ on the elements of the disjoint union $P \sqcup Q$ such that $s \leq t$ in $P \oplus Q$ if and only if at least one of the following conditions hold:
    \begin{enumerate}
        \item $s,t \in P$ and $s \leq_P t$, or
        \item $s,t \in Q$ and $s \leq_Q t$, or
        \item $s \in P$ and $t \in Q$. 
    \end{enumerate}
\end{defn}

The $n$-element chain will be denoted $C_n$ and the $k$-element antichain will be denoted $T_k$. In the language of ordinal sums, we can view $C_n$ as the ordinal sum of $n$ copies of $C_1$'s and we can view attaching $k$ minimal elements to a poset $P$ as the ordinal sum $T_k \oplus P$. Our main result in this section provides a way to compute the sorting generating function $f_{T_k \oplus P}(q)$ from $f_P(q)$.

Define a lower-triangular $n \times n$ matrix $X_n(k)$ whose $(i,j)$ entry $x_{ij}$ is given by
\[
x_{ij} \coloneqq \begin{cases}
          k!\binom{k+i-2}{k-1} & \text{if $i>j$,} \\
          k!\binom{k+i-1}{k} & \text{if $i=j$,} \\
          0 & \text{otherwise.}
          \end{cases}
\]

Recall that given a labeling on a poset, the \definition{standardization} of the restricted labeling on a subposet $Q$ shifts the labels to those from 1 to $|Q|$; see \cref{standard-def}.

\begin{thm}\label{thm:attach_k_minimal_gf}
    Let $P$ be an $n$-element poset and $f_P(q) = \sum_{i=0}^{n-1}a_iq^i$ be the sorting generating function of $P$. Write the sorting generating function of $T_k \oplus P$ as $f_{T_k \oplus P}(q) = \sum_{i=0}^{n+k-1}a'_iq^i$. 
 Let $v=(a_0,a_1,\dotsc,a_{n-1})^{\intercal}$ be the column vector of the coefficients of $f_P(q)$ and $v^{\prime} = (a_0^{\prime},a_1^{\prime},\dotsc,a_{n-1}^{\prime})^{\intercal}$ the column vector of the first $n$ coefficients of $f_{T_k \oplus P}(q)$. Then
    \begin{enumerate}
        \item $X_n(k)v = v^{\prime}$,
        \item $a_n^{\prime} = n! k!\binom{n+k-1}{k-1}$, and
        \item $a_i^{\prime} = 0$ for $i=n+1,n+2,\dotsc,n+k-1$.
    \end{enumerate}
\end{thm}
\begin{proof}
Let $x_1,x_2,\ldots,x_k$ be the elements of $T_k$. Since the roles of the $x_i$'s are symmetrical, it follows that permuting the labels of the $x_i$'s on any labeling $L \in \Lambda(T_k \oplus P)$ doesn't change $\order(L)$. Therefore, we will compute the number of labelings that satisfy $L(x_1) < L(x_2) < \cdots < L(x_k)$ and then multiply by $k!$.

Now, we will define a procedure that, given a labeling $L \in \Lambda(P)$ and a $k$-tuple of distinct numbers $I = (i_1, \ldots, i_k)\in [n+k]^k$, 
produces a labeling $L^I \in \Lambda(T_k \oplus P)$ such that $L^I(x_s) = i_s$ for $1 \le s \le k$. Since we are counting labelings where the labels of the $x_i$'s are increasing, we will assume that $i_1 < i_2 < \cdots < i_k$ for the rest of the proof.

To obtain $L^I$, first define labelings of $L^0, L^1, \ldots, L^k$ of $P$, where $L^0 \coloneqq L$ and for $s = 1, \ldots, k$, recursively define $L^s$ by
\[
L^{s}(x) \coloneqq \begin{cases}
          L^{s-1}(x)+1 & \text{if $L^{s-1}(x) \geq i_s$,} \\
          L^{s-1}(x) & \text{otherwise.}
          \end{cases}
\]

Then define $L^I$ on $T_k\oplus P$ by
\begin{equation}
    L^I(x) \coloneqq \begin{cases}
          i_s & \mbox{if $x = x_s$,} \\
          L^k(x) & \mbox{if $x\in P$.}
          \end{cases}
\end{equation}

In \cref{fig:LI}, we give an example of defining $L^I$ of $T_3 \oplus P$ on a 7-element poset $P$ and with $I=(2,4,7)$. The labeling of $P$ is given in the left figure, and the middle three figures illustrate the process mentioned above. The right figure is the resulting labeling $L^I$ of $T_3 \oplus P$.
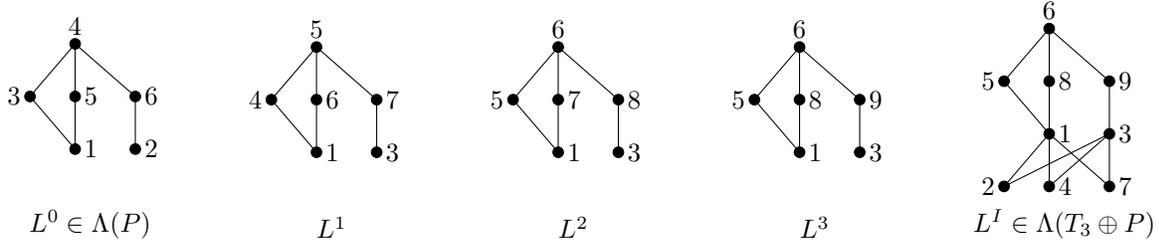
\begin{figure}[htb!]
    \centering
    \begin{subfigure}{0.18\textwidth}
        \begin{tikzpicture}
            \draw (0,1.4) -- (0,0.7) -- (0,0) -- (-0.6,0.7) -- (0,1.4) -- (0.8,0.7) -- (0.8,0);
            \filldraw (0,1.4) circle (2pt) node [anchor=south]{$4$};
            \filldraw (0,0.7) circle (2pt) node [anchor=west]{$5$};
            \filldraw (0,0) circle (2pt) node [anchor=west]{$1$};
            \filldraw (-0.6,0.7) circle (2pt) node [anchor=east]{$3$};
            \filldraw (0.8,0.7) circle (2pt) node [anchor=west]{$6$};
            \filldraw (0.8,0) circle (2pt) node [anchor=west]{$2$};
            \node at (0.2, -1) {$L^0 \in \Lambda(P)$};
        \end{tikzpicture}
    \end{subfigure}
    \begin{subfigure}{0.18\textwidth}
        \begin{tikzpicture}
            \draw (0,1.4) -- (0,0.7) -- (0,0) -- (-0.6,0.7) -- (0,1.4) -- (0.8,0.7) -- (0.8,0);
            \filldraw (0,1.4) circle (2pt) node [anchor=south]{$5$};
            \filldraw (0,0.7) circle (2pt) node [anchor=west]{$6$};
            \filldraw (0,0) circle (2pt) node [anchor=west]{$1$};
            \filldraw (-0.6,0.7) circle (2pt) node [anchor=east]{$4$};
            \filldraw (0.8,0.7) circle (2pt) node [anchor=west]{$7$};
            \filldraw (0.8,0) circle (2pt) node [anchor=west]{$3$};
            \node at (0.2, -1) {$L^1$};
        \end{tikzpicture}
    \end{subfigure}
    \begin{subfigure}{0.18\textwidth}
        \begin{tikzpicture}
            \draw (0,1.4) -- (0,0.7) -- (0,0) -- (-0.6,0.7) -- (0,1.4) -- (0.8,0.7) -- (0.8,0);
            \filldraw (0,1.4) circle (2pt) node [anchor=south]{$6$};
            \filldraw (0,0.7) circle (2pt) node [anchor=west]{$7$};
            \filldraw (0,0) circle (2pt) node [anchor=west]{$1$};
            \filldraw (-0.6,0.7) circle (2pt) node [anchor=east]{$5$};
            \filldraw (0.8,0.7) circle (2pt) node [anchor=west]{$8$};
            \filldraw (0.8,0) circle (2pt) node [anchor=west]{$3$};
            \node at (0.2, -1) {$L^2$};
        \end{tikzpicture}
    \end{subfigure}
    \begin{subfigure}{0.18\textwidth}
        \begin{tikzpicture}
            \draw (0,1.4) -- (0,0.7) -- (0,0) -- (-0.6,0.7) -- (0,1.4) -- (0.8,0.7) -- (0.8,0);
            \filldraw (0,1.4) circle (2pt) node [anchor=south]{$6$};
            \filldraw (0,0.7) circle (2pt) node [anchor=west]{$8$};
            \filldraw (0,0) circle (2pt) node [anchor=west]{$1$};
            \filldraw (-0.6,0.7) circle (2pt) node [anchor=east]{$5$};
            \filldraw (0.8,0.7) circle (2pt) node [anchor=west]{$9$};
            \filldraw (0.8,0) circle (2pt) node [anchor=west]{$3$};
            \node at (0.2, -1) {$L^3$};
        \end{tikzpicture}
    \end{subfigure}
    \begin{subfigure}{0.18\textwidth}
        \begin{tikzpicture}
            \draw (0,1.4) -- (0,0.7) -- (0,0) -- (-0.6,0.7) -- (0,1.4) -- (0.8,0.7) -- (0.8,0);
            \draw (-0.6,-0.7) -- (0,0) -- (0,-0.7) -- (0.8,0) -- (0.8,-0.7);
            \draw (0,0) -- (0.8,-0.7);
            \draw (0.8,0) -- (-0.6,-0.7);
            \filldraw (0,1.4) circle (2pt) node [anchor=south]{$6$};
            \filldraw (0,0.7) circle (2pt) node [anchor=west]{$8$};
            \filldraw (0,0) circle (2pt) node [anchor=west]{$1$};
            \filldraw (-0.6,0.7) circle (2pt) node [anchor=east]{$5$};
            \filldraw (0.8,0.7) circle (2pt) node [anchor=west]{$9$};
            \filldraw (0.8,0) circle (2pt) node [anchor=west]{$3$};
            \filldraw (0.8,-0.7) circle (2pt) node [anchor=west]{$7$};
            \filldraw (0,-0.7) circle (2pt) node [anchor=west]{$4$};
            \filldraw (-0.6,-0.7) circle (2pt) node [anchor=east]{$2$};
            \node at (0.2, -1.2) {$L^I \in \Lambda(T_3 \oplus P)$};
        \end{tikzpicture}
    \end{subfigure}
    \caption{Defining $L^I$ of $T_3 \oplus P$ with $I=(2,4,7)$.}
    \label{fig:LI}
\end{figure}

    One can check that at each step $s = 1, \ldots, k$ the standardization $\st(L^s)$ is precisely $L$. Therefore, the standardization of $L^I|_P$ is $\st(L^I|_P) = \st(L^k) = L$. In other words, $L^I$ is the unique labeling in $\Lambda(T_k \oplus P)$ that assigns the label $i_s$ to $x_s$ for $s = 1, \ldots, k$ and whose standardization when restricted to $P$ is $L$. As a consequence, the set of labelings $\Lambda(T_k \oplus P)$ can be partitioned as
    \begin{equation}
        \Lambda(T_k \oplus P) = \bigsqcup_{L \in \Lambda(P)} \left\{L^{[I]} : I \in \binom{[n]}{k} \right\},
        \label{eq:attach-k-minimum-elements-decomposition}
    \end{equation}
    where $L^{[I]}$ contains the labeling $L^I$ and the labelings obtained from $L^I$ by permuting all the labels of the $x_i$'s.

    Next, we proceed with the following two claims.
    
\noindent \textbf{Claim 1.} Given $L \in \Lambda(P)$ and $I = (i_1, \ldots, i_k)$, the standardization of $L^I|_P$ is preserved under a sequence of promotions:
    \begin{equation}
        \st((L_j^I)|_{P}) =  L_j\text{, for all $j \in \mathbb{Z}_{\geq 0}$.}
        \label{eq:standardization-preserved}
    \end{equation}
\begin{proof}[Proof of Claim 1.]
    We will show Claim $1$ by induction. When $j=0$, the identity holds by the definition of $L^I$. Suppose it holds for some $j$ and consider $L_{j+1}^I$. If $L_{j}^I(x_s)>1$ for all $s=1,2,\dotsc,k$, then these minimal elements $x_s$'s are not in the $(j+1)$-th promotion chain and the claim holds. On the other hand, if there exists an $s$ such that $L_{j}^I(x_s)=1$, then the $(j+1)$-th promotion begins at $x_s$. Since $x_s \leq x$ for all $x \in P$, the next element in the promotion chain is $(L_{j}^I)^{-1}(y)$, where $y=\min\{L_{j}^I(z) : z \in P\}$. This element is exactly $(L_{j}^I)^{-1}(y) = (L_j)^{-1}(1)$. From this point on, the rest of the promotion chain is the same in $L_j^I$ and $L_j$. Therefore, $\st((L_j^I)|_{P}) = L_j$ for all $j \in \mathbb{Z}_{\geq 0}$.
\end{proof}    
    
\noindent \textbf{Claim 2.} Given $L \in \Lambda(P)$ and $I = (i_1, \ldots, i_k)$, the order of $L^I$ is given by 
\begin{equation}\label{eq:order_LI}
    \order(L^I) = \max(i_k - k, \order(L)).
\end{equation}
\begin{proof}[Proof of Claim 2.]
    We observe that for some nonnegative integer $j$, $L_j^I$ is a natural labeling if and only if two conditions are satisfied:
    \begin{enumerate}
        \item the set of labels $\{L_j^I(x_1),\dotsc,L_j^I(x_k)\}$ is $[k]$, and
        \item $(L_j^I)|_P$ is a natural labeling.
    \end{enumerate}
    By \cref{eq:standardization-preserved}, the second condition is satisfied if and only if $j \ge \order(L)$. On the other hand, we show below that the first condition is satisfied if and only if $j \ge i_k-k$. 
    
    To see this, we notice that the first $i_1-1$ promotions only decrement the labels of $x_1, \ldots, x_k$. Let $\mathcal{S}_j \coloneqq \{L_j^I(x_1), \ldots, L_j^I(x_k)\}$ and let $s_j$ be the maximum value (possibly $0$) such that $[s_j] \subseteq \mathcal{S}_j$. Then the minimum label in $L_j^I|_P$ is $s_j+1$ and in the $(j+1)$-th promotion, $(L_j^I)^{-1}(s_j+1)$ is part of the promotion chain, so $\mathcal{S}_{j+1} = [s_j] \cup \{y - 1 : y \in \mathcal{S}_j \setminus [s_j]\}$. Note that $s_{j+1} > s_j$ if and only if $s_j + 1 \in \{y-1 : y \in \mathcal{S}_j \setminus [s_j]\}$. Thus, it follows by an inductive argument that $s_j \ge t$ if and only if $j \ge i_t - t$ which yields the desired result. Combining these two conditions implies that $\order(L^I) = \max(i_k - k, \order(L))$.
\end{proof}

    We are now ready to prove the first statement, in which we show that for $1 \le s \le n$, $k!$ times the number of labelings in $\Lambda(T_k \oplus P)$ with order $m-1$ is equal to the $m$th row of $X_n(k)v$. By \cref{eq:attach-k-minimum-elements-decomposition}, we can sum over all labelings $L \in \Lambda(P)$ and count the number of $I \in \binom{[n]}{k}$ such that $\order(L^I) = m-1$. We proceed by cases analysis of $\order(L)$.
    \begin{itemize}
        \item Suppose $\order(L) < m-1$. Then in order for $\order(L^I) = \max(i_k-k, \order(L)) = m-1$ to hold, it must be that $i_k-k = m-1$. Fixing $i_k = k+m-1$, there are $\binom{k+m-2}{k-1}$ ways to choose $i_1, \ldots, i_{k-1}$ such that $\order(L^I) = m-1$.

        \item Suppose $\order(L)=m-1$. Then in order for $\order(L^I) = \max(i_k-k, \order(L)) = m-1$ to hold, it must be that $i_k - k \le m-1$. Thus, $i_k \le k+m-1$ so there are $\binom{k+m-1}{k}$ ways to choose $i_1, \ldots, i_k$ such that $\order(L^I) = m-1$.

        \item Suppose $\order(L) > m-1$. Then $\order(L^I) = \max(i_k-k,\order(L)) > m-1$ so there are no choices of $I$ that yield $\order(L^I) = m-1$.
    \end{itemize}
    After multiplying by $k!$ to account for the fact that permuting the labels of $x_1, \ldots, x_k$ do not change the order of a labeling of $T_k \oplus P$, the first case yields the $(m,j)$ entry of $X_n(k)$ when $j < m$, the middle case yields the $(m,m)$ entry of $X_n(k)$, and the last case yields the $(m,j)$ entry of $X_n(k)$ when $j > m$. This completes the proof of the first statement.

    To prove the second statement, observe that since $\order(L) \le n-1$ for any $L \in \Lambda(P)$, then $\order(L^I) = \max(i_k-k, \order(L)) = n$ if and only if $i_k -k = n$. Fixing $i_k = k+n$, there are $\binom{n+k-1}{k-1}$ choices for $i_1, \ldots, i_{k-1}$, regardless of $\order(L)$. Multiplying by $k!$ to account for permuting the labels of $x_1, \ldots, x_k$ yields 
    \begin{equation*}
        a_n^{\prime} = k!\binom{n+k-1}{k-1}(a_0 + a_1 + \cdots + a_{n-1}) = n!k!\binom{n+k-1}{k-1}.
    \end{equation*}
    This completes the proof of the second statement.

    Finally to prove the last statement, first observe that $i_k \le k+n$ since there are only $k+n$ elements in $T_k \oplus P$. Thus, $i_k - k \le n$. In addition, any labeling $L \in \Lambda(P)$ has $\order(L) \le n-1$. It follows that $\order(L^I) \le n$ for any choice of $L \in \Lambda(P)$ and $I \in \binom{[n+k]}{k}$. Thus, there do not exist labelings $T_k \oplus P$ with order greater than $n$ and hence $a_i^{\prime} = 0$ for $i=n+1,n+2,\dotsc,n+k-1$. This completes the proof of the last statement.
\end{proof}

We would like to point out that if $k\ge 2$, then $T_k\oplus P$ has no tangled labelings.

\begin{example}\label{ex:attachgf}
    Let $P$ be as in \cref{ex:fg}. The sorting generating function of $P$ is given by $f_P(q) = 2+4q$. Let $v$ be the column vector $(2,4,0)^{\intercal}$. We show below how to obtain the sorting generating function of posets shown in \cref{fig:attachposet} from \cref{thm:attach_k_minimal_gf}. 

    For $T_1 \oplus P$,
    \begin{equation*}
        X_3(1)v = \begin{pmatrix}
            1 & 0 & 0 \\
            1 & 2 & 0 \\
            1 & 1 & 3
        \end{pmatrix} \begin{pmatrix}
            2 \\
            4 \\
            0 
        \end{pmatrix} = \begin{pmatrix}
            2 \\
            10 \\
            6 
        \end{pmatrix} \text{ and } a_3^{\prime} = 1(2+4+0)=6.
    \end{equation*}
    Then $f_{T_1 \oplus P}(q) = 2+10q+6q^2+6q^3$. For $T_2 \oplus P$,
    \begin{equation*}
        X_3(2)v = \begin{pmatrix}
            2 & 0 & 0 \\
            4 & 6 & 0 \\
            6 & 6 & 12
        \end{pmatrix} \begin{pmatrix}
            2 \\
            4 \\
            0 
        \end{pmatrix} = \begin{pmatrix}
            4 \\
            32 \\
            36 
        \end{pmatrix} \text{ and } a_3^{\prime} = 8(2+4+0)=48.
    \end{equation*}
    Then $f_{T_2 \oplus P}(q) = 4+32q+36q^2+48q^3$. Finally, for $T_3 \oplus P$,
    \begin{equation*}
        X_3(3)v = \begin{pmatrix}
            6 & 0 & 0 \\
            18 & 24 & 0 \\
            36 & 36 & 60
        \end{pmatrix} \begin{pmatrix}
            2 \\
            4 \\
            0 
        \end{pmatrix} = \begin{pmatrix}
            12 \\
            132 \\
            216 
        \end{pmatrix} \text{ and } a_3^{\prime} = 60(2+4+0)=360.
    \end{equation*}
    Then $f_{T_3 \oplus P}(q) = 12+132q+216q^2+360q^3$.
\begin{figure}[htb!]
    \centering
    \begin{subfigure}{0.18\textwidth}
        \begin{tikzpicture}
        \draw (-0.55,0) -- (0,1) --  (0.55,0);
        \filldraw[black] (-0.55,0) circle (2pt) node [anchor=east] {};
        \filldraw[black] (0,1) circle (2pt) node [anchor=south] {};
        \filldraw[black] (0.55,0) circle (2pt) node [anchor=west] {};
        \node at (0,-0.5) {$P$};
        \end{tikzpicture}
    \end{subfigure} 
    \begin{subfigure}{0.18\textwidth}
        \begin{tikzpicture}
        \draw (-0.55,0) -- (0,1) --  (0.55,0) -- (0,-1) -- (-0.55,0);
        \filldraw[black] (-0.55,0) circle (2pt) node [anchor=east] {};
        \filldraw[black] (0,1) circle (2pt) node [anchor=south] {};
        \filldraw[black] (0.55,0) circle (2pt) node [anchor=west] {};
        \filldraw[black] (0,-1) circle (2pt) node [anchor=west] {};
        \node at (0,-1.5) {$T_1 \oplus P$};
        \end{tikzpicture}
    \end{subfigure} 
    \begin{subfigure}{0.18\textwidth}
        \begin{tikzpicture}
        \draw (0.55,-1) -- (-0.55,0) -- (0,1) -- (0.55,0) -- (-0.55,-1);
        \draw (-0.55,0) -- (-0.55,-1);
        \draw (0.55,0) -- (0.55,-1); 
        \filldraw[black] (-0.55,0) circle (2pt) node [anchor=east] {};
        \filldraw[black] (0,1) circle (2pt) node [anchor=south] {};
        \filldraw[black] (0.55,0) circle (2pt) node [anchor=west] {};
        \filldraw[black] (-0.55,-1) circle (2pt) node [anchor=south] {};
        \filldraw[black] (0.55,-1) circle (2pt) node [anchor=west] {};
        \node at (0,-1.5) {$T_2 \oplus P$};
        \end{tikzpicture}
    \end{subfigure} 
    \begin{subfigure}{0.18\textwidth}
        \begin{tikzpicture}
        \draw (1.1,-1) -- (-0.55,0) -- (0,1) -- (0.55,0) -- (-1.1,-1);
        \draw (-0.55,0) -- (-1.1,-1);
        \draw (0.55,0) -- (1.1,-1); 
        \draw (-0.55,0) -- (0,-1); 
        \draw (0.55,0) -- (0,-1); 
        \filldraw[black] (-0.55,0) circle (2pt) node [anchor=east] {};
        \filldraw[black] (0,1) circle (2pt) node [anchor=south] {};
        \filldraw[black] (0.55,0) circle (2pt) node [anchor=west] {};
        \filldraw[black] (-1.1,-1) circle (2pt) node [anchor=south] {};
        \filldraw[black] (0,-1) circle (2pt) node [anchor=west] {};
        \filldraw[black] (1.1,-1) circle (2pt) node [anchor=west] {};
        \node at (0,-1.5) {$T_3 \oplus P$};
        \end{tikzpicture}
    \end{subfigure} 
    \caption{The posets obtained from $P$ by attaching $1,2$ and $3$ minimal elements.}
    \label{fig:attachposet}
\end{figure}
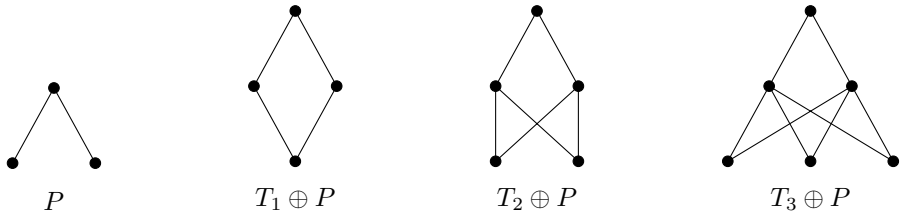
\end{example}

An analogous result for the cumulative generating function $g_{T_k \oplus P}(q)$ is stated below. 
\begin{thm}\label{thm:attach_k_element2}
    Let $P$ be an $n$-element poset and  $g_P(q) = \sum_{i=0}^{n-1}b_iq^i$ the cumulative generating function of $P$. Assume $g_{T_k \oplus P}(q) = \sum_{i=0}^{n+k-1} b_i^{\prime} q^i$. Let $w=(b_0,b_1,\dotsc,b_{n-1})^{\intercal}$ be the column vector of the coefficients of $g_P(q)$ and $w^{\prime} = (b_0^{\prime},b_1^{\prime},\dotsc,b_{n-1}^{\prime})^{\intercal}$ be the column vector of the first $n$ coefficients of $g_{T_k \oplus P}(q)$. Then
     \begin{enumerate}
        \item $Y_n(k)w = w^{\prime}$, where $Y_n(k)$ is the $n \times n$ diagonal matrix, the $i$th diagonal entry given by $\frac{(k+i-1)!}{(i-1)!}$.
        \item $b_i^{\prime} = (n+k)!$ for $i=n,n+1,\dotsc,n+k-1$.    
    \end{enumerate}   
\end{thm}
\begin{proof}
    Let $R_n$ be the lower triangular matrix of size $n$ whose lower triangular entries (including the diagonal entries) are $1$. If $v=(a_0,a_1,\dotsc,a_{n-1})^{\intercal}$ is the column vector of the coefficients of $f_P(q)$, then it is easy to see that $R_nv=w$. One can also check that $Y_n(k)R_n = R_nX_n(k)$. 

    By \cref{thm:attach_k_minimal_gf}, the first part of the statement follows from the identities below. 
    \begin{equation*}
        Y_n(k)w = Y_n(k)R_nv = R_nX_n(k)v = R_nv^{\prime} = w^{\prime}.
    \end{equation*}
    Since $a_i^{\prime} = 0$ for $i=n+1,n+2,\dotsc,n+k-1$, this implies that $b_i^{\prime} = (n+k)!$ for $i=n,n+1,\dotsc,n+k-1$.
\end{proof}

We close this section with a special family of posets which are obtained from a given $n$-element poset $P$ by attaching the chain with $\ell$ elements below $P$, that is, $\left( \bigoplus_{i=1}^{\ell} T_1 \right) \oplus P$. For convenience, we denote it by $P^{(\ell)}$. Note that $P^{(\ell)}$ has $n+\ell$ elements.

We assume that the sorting and cumulative generating functions of $P^{(\ell)}$ are written as $f_{P^{(\ell)}}(q) = \sum_{i=0}^{n+\ell-1} a_{i}^{(\ell)} q^i$ and $g_{P^{(\ell)}}(q) = \sum_{i=0}^{n+\ell-1} b_{i}^{(\ell)}q^i$, respectively. Two propositions are stated below.
\begin{prop}\label{prop:Pk_gf_b}
    Let $P$ be an $n$-element poset and $P^{(\ell)}$ the poset obtained from $P$ by attaching the chain with $\ell$ elements below $P$. The last $\ell+1$ coefficients of the cumulative generating function $g_{P^{(\ell)}}(q)$ are given by
    \begin{equation}\label{eq.coeff_of_b}
        b_{n+\ell-(r+1)}^{(\ell)} = (n+\ell-r)^{r}(n+\ell-r)!, 
    \end{equation}
    for $0 \leq r \leq \ell$.
    
    Moreover, $P^{(\ell)}$ satisfies \cref{conj:main-conjecture} if and only if 
    \begin{equation}\label{eq.lower_bound_of_b}
        b_{n-2}^{(\ell)} \geq (n-1)^{\ell+1}(n-1)!. 
    \end{equation}  
\end{prop}
\begin{proof}    
   Applying \cref{thm:attach_k_element2} with $k=1$ repeatedly, we obtain
   \begin{equation*}
       b_{n+\ell-(r+1)}^{(\ell)} = (n+\ell-r)b_{n+\ell-(r+1)}^{(\ell-1)}  = (n+\ell-r)^2 b_{n+\ell-(r+1)}^{(\ell-2)} = \cdots = (n+\ell-r)^t b_{n+\ell-(r+1)}^{(\ell-t)},
   \end{equation*}
   for $0 \leq r \leq \ell$ and for some non-negative integer $t$. When $(n+\ell-(r+1)) - (\ell-t) = n-1$, that is, when $t=r$, $b_{n+\ell-(r+1)}^{(\ell-t)}$ is the leading coefficient of the $g_{P^{(\ell-t)}}(q)$. So, $b_{n+\ell-(r+1)}^{(\ell-t)} = (n+\ell-r)!$. Therefore, $b_{n+\ell-(r+1)}^{(\ell)} = (n+\ell-r)^{r}(n+\ell-r)!$.

   \cref{conj:main-conjecture} implies that $a_{n-1} \leq (n-1)!$. Since $b_{n-1} = b_{n-2} + a_{n-1}$,
   \begin{equation*}
       b_{n-2} = b_{n-1} - a_{n-1} \geq n! - (n-1)! = (n-1)(n-1)!.
   \end{equation*}
   We again apply \cref{thm:attach_k_element2} with $k=1$ repeatedly, then
   \begin{equation*}
       b_{n-2}^{(\ell)} = (n-1) b_{n-2}^{(\ell-1)} = \cdots = (n-1)^{\ell} b_{n-2} \geq (n-1)^{\ell+1}(n-1)!.
   \end{equation*}
    The converse statement is argued in a similar way and will be omitted here.
\end{proof}
We then state below the counterpart result of Proposition \ref{prop:Pk_gf_b}.
\begin{prop}\label{prop:Pk_gf_a}
    Let $P$ be an $n$-element poset. For $0 \leq r \leq \ell-1$, the number of $r$-tangled labelings of $P^{(\ell)}$ is given by
    \begin{equation}\label{eq.r-tangled_labelings}
        a_{n+\ell-(r+1)}^{(\ell)} = \left((n+\ell-r)^{r+1}-(n+\ell-(r+1))^{r+1} \right) (n+\ell-(r+1))!.
    \end{equation}
    Moreover, $P^{(\ell)}$ satisfies \cref{conj:main-conjecture} if and only if
    \begin{equation}\label{eq.upper_bound_of_a}
        a_{n-1}^{(\ell)} \leq \left(n^{\ell+1}-(n-1)^{\ell+1} \right)(n-1)!. 
    \end{equation}
\end{prop}
\begin{proof}
    Notice that $b_{n+\ell-(r+1)}^{(\ell)} - b_{n+\ell-(r+2)}^{(\ell)} = a_{n+\ell-(r+1)}^{(\ell)}$ for $0 \leq r \leq \ell-1$. Then \cref{eq.r-tangled_labelings} follows immediately from \cref{eq.coeff_of_b}. 
    
    By \cref{eq.coeff_of_b} with $r=\ell$, $b_{n-1}^{(\ell)} = n^{\ell}n!$. Then \cref{eq.upper_bound_of_a} is obtained from $a_{n-1}^{(\ell)} = b_{n-1}^{(\ell)} - b_{n-2}^{(\ell)}$ and \cref{eq.lower_bound_of_b}. The converse statement can be argued similarly and is omitted here.
\end{proof}

We next show that our poset $P^{(\ell)}$ satisfies \cite[Conjecture 23]{hodges-2022}. This conjecture states that for an $n$-element poset $P$, the number of labelings $L \in \Lambda(P)$ such that $L_{n-3} \notin \mathcal{L}(P)$ has an upper bound $3(n-1)!$.
\begin{cor}\label{cor:(quasi)-tangled_labelings}
    Let $P$ be an $n$-element poset and $\ell \geq 1$. The number of labelings $L \in \Lambda(P^{(\ell)})$ such that $L_{n+\ell-3} \notin \mathcal{L}(P^{(\ell)})$, that is, the total number of tangled and quasi-tangled labelings of $P^{(\ell)}$, equals
    \begin{equation*}
            3(n+\ell-1)! - (n+\ell-2)! \le 3(n+\ell-1)!,  
    \end{equation*}
\end{cor}
\begin{proof}
    By \cref{prop:Pk_gf_a} with $r=0$, the number of tangled labelings of $P^{(\ell)}$ is          \begin{equation*}
        a_{n+\ell-1}^{(\ell)} = (n+\ell-1)!.
    \end{equation*}
    Take $r=1$ in \cref{prop:Pk_gf_a}, we obtain the number of quasi-tangled labelings of $P^{(\ell)}$, which is given by
     \begin{align*}
       a_{n+\ell-2}^{(\ell)} & =
        \left( (n+\ell-1)^{2} - (n+\ell-2)^{2}\right) (n+\ell-2)! \\
        & = \left( 2(n+\ell-1) -1 \right)(n+\ell-2)! \\
        & = 2(n+\ell-1)! - (n+\ell-2)!.
    \end{align*}
    Summing these two numbers gives the desired result. 
\end{proof}
\begin{remark}\label{rmk:(quasi)_tangled_p(l)}
    Let $P$ be an $n$-element poset. We are able to give a simple and unified proof of some results given by Defant and Kravitz in \cite{defant-kravitz-2023} and by Hodges in \cite{hodges-2022}.
    \begin{itemize}
        \item Take $\ell=1$, the poset $P^{(1)}$ has one minimal element. By \cref{prop:Pk_gf_a} with $r=0$, the number of tangled labelings of $P^{(1)}$ is given by
        \begin{equation*}
            a_{n}^{(1)} = \left(n+1 - n\right) n! = n!.
        \end{equation*}
        This gives an alternative proof of \cite[Corollary 3.7]{defant-kravitz-2023} (for a connected poset).
        
        \item Take $\ell=2$, the poset $P^{(2)}$ has one minimal element and this minimal element has exactly one parent. By \cref{prop:Pk_gf_a} with $r=1$, the number of quasi-tangled labelings of $P^{(2)}$ is given by
        \begin{equation*}
            a_{n}^{(2)} = \left( (n+1)^{2} - n^{2}\right) n! = (2n+1)n! = 2(n+1)!-n!.
        \end{equation*}
         This gives a simpler proof of \cite[Corollary 10]{hodges-2022}.
    \end{itemize}
\end{remark}

\section{Ordinal sum of antichains}\label{sec:ordinalsum}

In this section, we consider a family of posets consisting of the \definition{ordinal sum of antichains}. Let $C=(c_1,c_2,\dotsc,c_r)$ be an ordered sequence of $r$ positive integers. We write $P_C = \bigoplus_{i=1}^{r}T_{c_i}$ for the ordinal sum of antichains of $C$. We completely determine the cumulative generating function of this family of posets. We also show various properties and a poset structure of its cumulative generating function.

The cumulative generating function of the $k$-element antichain $T_k$ is $g_{T_k}(q) = k!(1+q+q^2+\cdots+q^{k-1})$. To find $g_{P_C}(q)$, we start from the antichain $T_{c_1}$ and let $w=(c_1!,\dotsc,c_1!)^{\intercal}$ be the column vector consisting of the coefficients of $g_{T_{c_1}}(q)$. We next attach $c_2$ minimal elements to $T_{c_1}$; the cumulative generating function $g_{T_{c_1} \oplus T_{c_2}}(q)$ is obtained by \cref{thm:attach_k_element2}. Recall that $Y_{c_1}(c_2)$ denotes the $c_1 \times c_1$ diagonal matrix whose $i$th diagonal entry is given by $\frac{(c_2+i-1)!}{(i-1)!}$. The matrix multiplication $Y_{c_1}(c_2)w$ gives the first $c_1$ coefficients of $g_{T_{c_1} \oplus T_{c_2}}(q)$ and the rest of coefficients are given by $(c_1+c_2)!$. As a consequence, we can obtain $g_{P_C}$ by applying \cref{thm:attach_k_element2} repeatedly in this way. The explicit formula of $g_{P_C}(q)$ is summarized in the following proposition.
\begin{prop}\label{prop:ordinalsumg}
    Let $P_C = \bigoplus_{i=1}^{r}T_{c_i}$ be the ordinal sum of antichains of $C$, where $C=(c_1,c_2,\dotsc,c_r)$ is an ordered sequence of $r$ positive integers. 
    Write $g_{P_C}(q) = \sum_{s=0}^{c_1+\cdots+c_r-1}b_s q^s$ for the cumulative generating function of $P_C$.
    For each $0 \le s < c_1 + \cdots + c_r$, let $j \in [r]$ be the unique integer such that $$\sum_{k=1}^{j-1}c_k  \leq s < \sum_{k=1}^j c_k.$$ Then
    \begin{equation}\label{eq:ordinal_sum_antichains_b}
        b_s = (c_1+c_2+\cdots+c_j)! \prod_{m=j+1}^{r} \frac{(c_m+s)!}{s!}.
    \end{equation}
\end{prop}

We now present the following symmetry property for the poset $B_{n,k} = T_n \oplus C_{k+1}$, where $C_{k+1}$ is the chain of $k+1$ elements and $n,k \in \mathbb{Z}_{\geq 0}$. This poset is sometimes called a \definition{broom}.
\begin{prop}\label{prop:broom_gf}
    Let $n,k \in \mathbb{B}_{\geq 0}$. Write $f_{B_{n,k}}(q) = \sum_{s=0}^{n+k}a_s(n,k)q^s$ for the sorting generating function of $B_{n,k}$. Then
    \begin{equation}\label{eq:broom_gf}
        a_s(n,k) = \begin{cases}
            (n+s)!(s+1)^{k+1-s} - (n+s-1)!s^{k+2-s}, & \text{for $s=0,1,\dotsc,k+1$,} \\
            0, & \text{for $s=k+2,k+3,\dotsc,n+k$.}
        \end{cases}
    \end{equation}
    In particular, we have the symmetry property
    \begin{equation}\label{eq:broomsymmetric}
        a_k(n,k) = a_n(k,n), \text{ for $0 \leq n \leq k$}.
    \end{equation}
\end{prop}
\begin{proof} 
    By \cref{prop:ordinalsumg} with $c_1=c_2=\cdots=c_{k+1}=1$ and $c_{k+2}=n$, the cumulative generating function of $T_n \oplus C_{k+1}$ is given by $g_{T_n \oplus C_{k+1}}(q) = \sum_{s=0}^{n+k} b_s(n,k)q^s$, where 
    \begin{equation*}
        b_s(n,k) = (s+1)!(s+1)^{k-s}\frac{(n+s)!}{s!} = (n+s)!(s+1)^{k+1-s},
    \end{equation*}
    for $s=0,1,\dotsc,k$. We also have $b_s(n,k) = (n+k+1)!$ for $s = k+1,k+2,\dotsc,n+k$.
    
    Then \cref{eq:broom_gf} follows immediately from the fact that $a_s(n,k) = b_{s}(n,k) - b_{s-1}(n,k)$. The symmetry property (\cref{eq:broomsymmetric}) can be verified directly using \cref{eq:broom_gf}. This completes the proof of \cref{prop:broom_gf}.
\end{proof}

We next study problems proposed by Defant and Kravitz \cite{defant2020promotionsorting}\footnote{The problems are stated as Conjecture 5.2 and Problem 5.3 in their preprint, but not in the published version \cite{defant-kravitz-2023}.}.
Given an $n$-element poset $P$, are the coefficients of the sorting generating function $f_{P}(q)$ and the cumulative generating function $g_P(q)$ unimodal or log-concave? We prove that the coefficients of the cumulative generating function are log-concave for the ordinal sum of antichains and provide a counterexample to the conjecture that the coefficients of the sorting generating function of a general poset are unimodal.

Recall that a sequence of real numbers $(a_i)_{i=0}^{n}$ is called \definition{unimodal} if there is an index $j$ such that $a_0 \leq a_1 \leq a_2 \leq \cdots \leq a_j \geq a_{j+1} \geq \cdots \geq a_n$. We say this sequence is \definition{log-concave} if $a_i^2 \geq a_{i-1}a_{i+1}$ for $i=1,2,\dotsc,n-1$. Note that a positive sequence is log-concave implies that this sequence is unimodal.

We show below that the coefficients of the cumulative generating function of $P_C$ are log-concave.
\begin{prop}\label{prop:ordinal_sum_antichain_log-concave}
Let $P_C = \bigoplus_{i=1}^{r}T_{c_i}$ be the ordinal sum of antichains of $C$, where $C=(c_1,c_2,\dotsc,c_r)$ is a sequence of $r$ positive integers. Let $g_{P_C}(q) = \sum_{s=0}^{c_1+\cdots+c_r-1}b_s q^s$ be the cumulative generating function of $P_C$. Then the sequence $(b_s)_{s=0}^{c_1+\dotsc+c_r-1}$ is log-concave.
\end{prop}
\begin{proof}
   We will show that $\frac{b_{s}^2}{b_{s-1}b_{s+1}} \geq 1$ for $s=1,2,\dotsc,c_1 + \cdots + c_r - 2$ by direct computation using \cref{eq:ordinal_sum_antichains_b}. For $j = 1, 2, \ldots, r$, let $\mathcal{I}_j = \{s : \sum_{k=1}^{j-1}c_k  \leq s < \sum_{k=1}^j c_k\}$. The proof is based on the following four cases of the index $s$. We present the calculation for the first two cases below; the other two cases can be proved similarly and we leave them to the reader.
   
    \textbf{Case 1:} $s-1,s,s+1 \in \mathcal{I}_j$ for some $j$. In this case
        \begin{align*}
            \frac{b_{s}^2}{b_{s-1}b_{s+1}} & = \frac{\left( (c_1+\dotsc+c_j)! \prod_{m=j+1}^{r} \frac{(c_m+s)!}{s!} \right)^2}{\left( (c_1+\dotsc+c_j)! \right)^2 \prod_{m=j+1}^{r} \frac{(c_m+s-1)!(c_m+s+1)!}{(s-1)!(s+1)!}} \\
            &= \prod_{m=j+1}^{r}\frac{(s+1)(c_m+s)}{s(c_m+s+1)}\\
            &= \prod_{m=j+1}^{r}\frac{sc_m+s^2+c_m+s}{sc_m+s^2+s} \geq 1,
        \end{align*}
        since $c_m$ and $s$ are positive integers and thus the denominator is always smaller than the numerator.

   \textbf{Case 2:} $s-1,s \in \mathcal{I}_j$ and $s+1 \in \mathcal{I}_{j+1}$ for some $j$. In this case, $s=\sum_{k=1}^{j}c_k-1$, and 
        \begin{align*}
            \frac{b_{s}^2}{b_{s-1}b_{s+1}} & = \frac{\left( (c_1+\dotsc+c_j)! \prod_{m=j+1}^{r} \frac{(c_m+s)!}{s!} \right)^2}{\left( (c_1+\dotsc+c_j)! \prod_{m=j+1}^{r} \frac{(c_m+s-1)!}{(s-1)!} \right) \left((c_1+\dotsc+c_{j+1})! \prod_{m=j+2}^{r} \frac{(c_m+s+1)!}{(s+1)!} \right)}\\
            &= \frac{(c_1+\dotsc+c_j)!}{(c_1+\dotsc+c_{j+1})!} \cdot \frac{(c_{j+1}+s)!(c_{j+1}+s)!(s-1)!}{(c_{j+1}+s-1)!s!s!} \cdot \prod_{m=j+2}^{r}\frac{(s+1)(c_m+s)}{s(c_m+s+1)}  \\   
            &= \frac{(c_1+\dotsc+c_j)!}{(c_1+\dotsc+c_{j+1})!}\frac{(c_{j+1}+s)!\cdot (c_{j+1}+s)}{s! \cdot s} \cdot \prod_{m=j+2}^{r}\frac{(s+1)(c_m+s)}{s(c_m+s+1)}\\
            &= \frac{(s+1)!}{(s+1+c_{j+1})!}\frac{(c_{j+1}+s)!\cdot (c_{j+1}+s)}{s! \cdot s} \cdot \prod_{m=j+2}^{r}\frac{(s+1)(c_m+s)}{s(c_m+s+1)}\\
            & = \frac{(s+1)(c_{j+1}+s)}{s(c_{j+1}+s+1)} \cdot \prod_{m=j+2}^{r}\frac{(s+1)(c_m+s)}{s(c_m+s+1)} \ge 1
        \end{align*}
        by similar reasoning as in Case 1.

        We omit the calculation of showing $\frac{b_{s}^2}{b_{s-1}b_{s+1}} \geq 1$ for the last two cases, since they can be proved similarly.
        
   \textbf{Case 3:} $s-1 \in \mathcal{I}_j$ and $s, s+1 \in \mathcal{I}_{j+1}$ for some $j$. In this case, $s=\sum_{k=1}^{j}c_k$. 
       
   \textbf{Case 4:} $s-1 \in \mathcal{I}_j$, $s \in \mathcal{I}_{j+1}$ and $s+1 \in \mathcal{I}_{j+2}$ for some $j$. In this case, $c_{j+1}=1$ and $s= \sum_{k=1}^{j+1}c_k$. 
\end{proof}
\begin{remark}\label{rmk:false}
    For the poset $P=T_2 \oplus T_2 \oplus T_2$ the sorting generating function is $f_P(q) = 8 + 64q + 216q^2 + 192q^3 + 240q^4$ and the cumulative generating function is $g_P(q) = 8 + 72q + 288q^2 + 480q^3 + 720q^4 + 720q^5$. One can see that the coefficients of $f_P(q)$ are not unimodal, which gives a counterexample to \cite[Conjecture 29]{hodges-2022} (see also \cite[Conjecture 5.2]{defant2020promotionsorting}). One can also check that the coefficients of $g_P(q)$ are log-concave. 
\end{remark}

We close this section with a new direction for studying the cumulative generating function of the ordinal sum of antichains $P_C$. One can ask: how do the cumulative generating functions $g_{P_C}(q)$ and $g_{P_{C^{\prime}}}(q)$ compare when $C^{\prime}$ is a permutation of elements of $C$? Given an ordered sequence of $r$ distinct positive integers $C=(c_1,c_2,\dotsc,c_r)$ and a permutation $\pi$ in the symmetric group on $r$ elements $\mathfrak{S}_r$, define $\pi(C) = (c_{\pi(1)},c_{\pi(2)}, \dotsc,c_{\pi(r)} )$. The collection of the coefficients of the cumulative generating function of $P_{\pi(C)}$ for all $\pi \in \mathfrak{S}_r$ is defined to be
\begin{equation*}
    \mathbf{B}(C) = \{ \mathbf{b_{\pi}} = (b_0,b_1,\dotsc,b_{|C|-1} ) : \mathbf{b_{\pi}} \text{ is the sequence of coefficients of $g_{P_{\pi(C)}}(q)$, $\pi \in \mathfrak{S}_r$} \},
\end{equation*}
where $|C|=\sum_{i=1}^{r}c_i$. A natural partial order on $\mathbf{B}(C)$ is given by the following dominance relation.
\begin{defn}
    For a pair of integer sequences $\mathbf{b} = (b_0,b_1,\dotsc,b_n)$ and $\mathbf{b^{\prime}} = (b_0^{\prime},b_1^{\prime},\dotsc,b_n^{\prime})$, we say $\mathbf{b^{\prime}}$ \definition{dominates} $\mathbf{b}$, denoted by $\mathbf{b} \preceq \mathbf{b^{\prime}}$, if $b_i \leq b_i^{\prime}$ for $i=0,1,\dotsc,n$.
\end{defn}
If $\mathbf{b}$ and $\mathbf{b^{\prime}}$ denote the coefficients of the cumulative generating function of $P$ and $P^{\prime}$ respectively, then the relation $\mathbf{b} \preceq \mathbf{b^{\prime}}$ can be interpreted as saying that the labelings of $P^{\prime}$ require fewer promotions to be sorted compared to those of $P$. It is easy to check that $\preceq$ is a partial order on the set $\mathbf{B}(C)$.
\begin{example}
    For $C=(1,2,3)$, the cumulative generating functions $P_{\pi(C)}$ for $\pi \in \mathfrak{S}_3$ are computed and their coefficients listed below:
    \begin{align*}
        \mathbf{b}_{123} &= (12,144,360,720,720,720), &\mathbf{b}_{132} &= (12,144,288,480,720,720),
        &\mathbf{b}_{213} &= (12,96,360,720,720,720), \\
        \mathbf{b}_{312} &= (12,72,216,480,720,720),
        &\mathbf{b}_{231} &= (12,96,360,480,600,720),
        &\mathbf{b}_{321} &= (12,72,216,480,600,720).
    \end{align*}
    The Hasse diagram of $(\mathbf{B}(C),\preceq)$ is shown in the left of \cref{fig:Hassediagram}. Observe that the subgraph consisting of all the black edges forms the Hasse diagram of the dual to the weak order on $\mathfrak{S}_3$ (see for instance \cite[Exercises 3.183 and 3.185]{ECI} for the definition of weak and strong order on $\mathfrak{S}_n$). The red edge ($\mathbf{b}_{312} \preceq \mathbf{b}_{213}$) shows a new cover relation which does not occur in the weak order on $\mathfrak{S}_3$. 

    Moreover, we draw the Hasse diagram of $(\mathbf{B}(C),\preceq)$ where $C=(1,2,3,4)$ in the right picture of \cref{fig:Hassediagram}. Similarly, the subgraph consisting of black edges forms the Hasse diagram of the dual to the weak order on $\mathfrak{S}_4$ while the red edges show new cover relations in our poset structure compared to the weak order of $\mathfrak{S}_4$. We formulate this observation more generally in the following theorem.
\end{example}
\begin{figure}[htb!]
    \centering
    \begin{subfigure}[t]{0.3\textwidth}
        \centering
        \includegraphics[scale = 0.42]{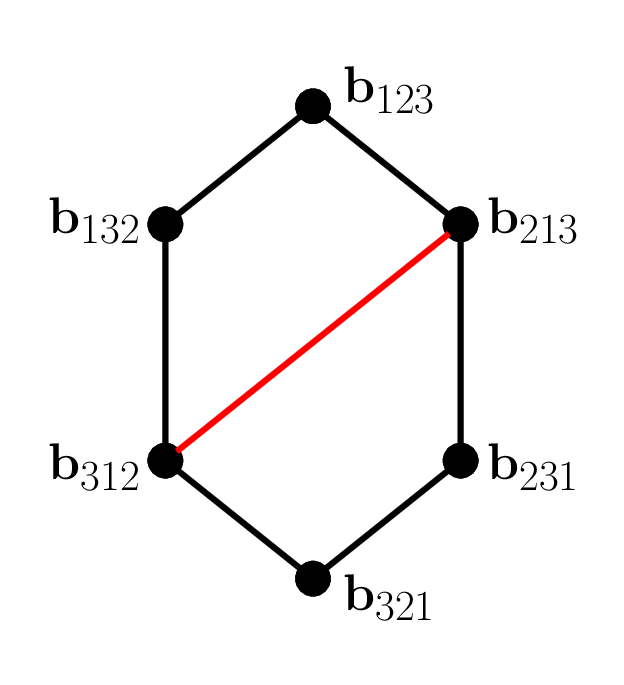}
    \end{subfigure}%
    \begin{subfigure}[t]{0.7\textwidth}
        \centering
        \includegraphics[scale = 0.42]{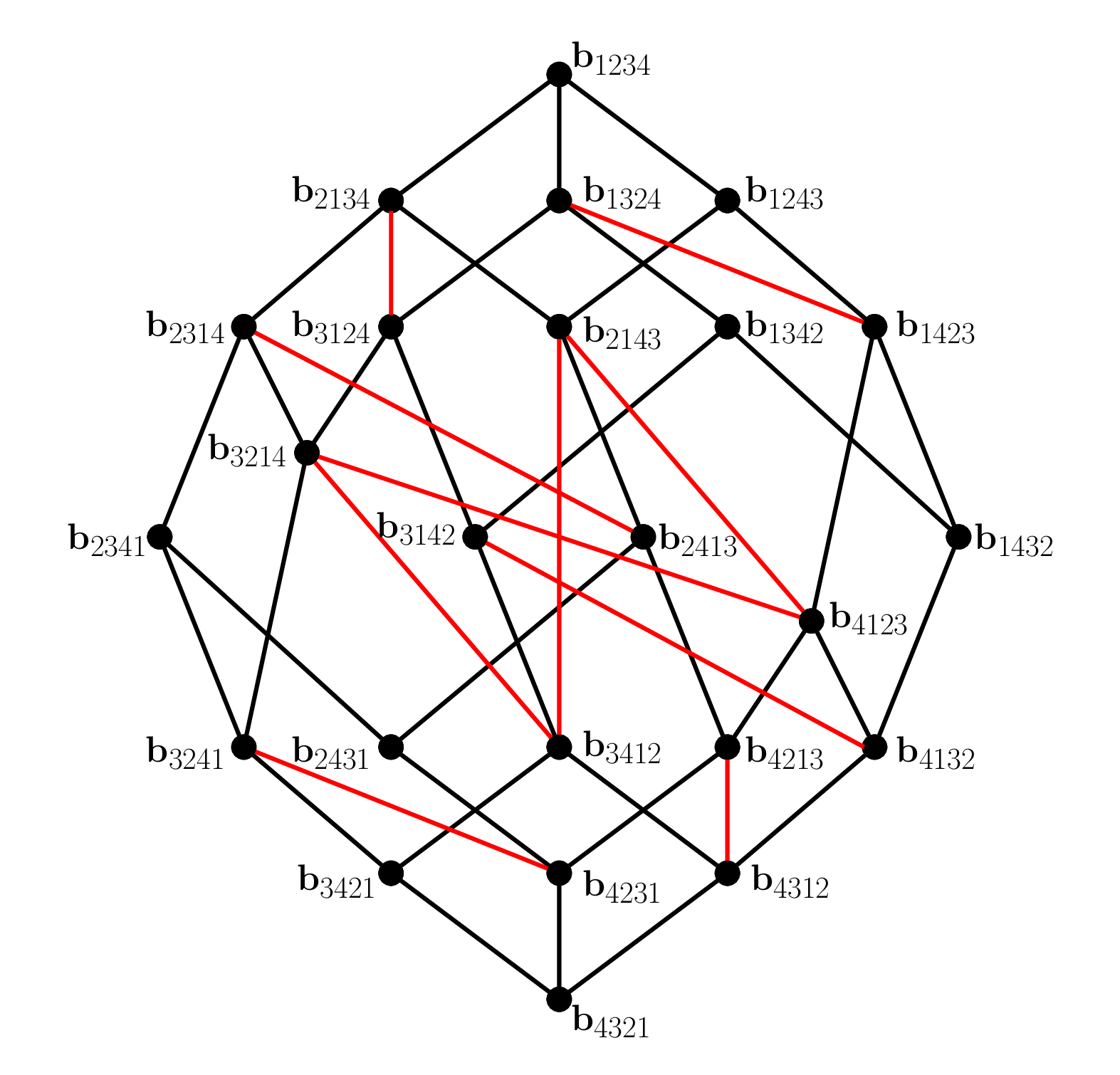}
    \end{subfigure}%
    \caption{The Hasse diagram of $(\mathbf{B}(C),\preceq)$ where $C=(1,2,3)$ (left) and $C=(1,2,3,4)$ (right), which contains a subposet (shown as a subgraph consisting of all the black edges) that is isomorphic to the weak order of $\mathfrak{S}_3$ (left) and $\mathfrak{S}_4$ (right). The new cover relations in our poset structure compared to the weak order of $\mathfrak{S}_3$ (left) and $\mathfrak{S}_4$ (right) are drawn in red.}
    \label{fig:Hassediagram}
\end{figure}
\begin{thm}\label{thm:antichain_Bruhat_order}
    Given an ordered sequence of $r$ distinct positive integers $C=(c_1,c_2,\dotsc,c_r)$. Let $$\mathbf{B}(C) = \{ \mathbf{b_{\pi}} = (b_0,b_1,\dotsc,b_{|C|-1}) \mid \mathbf{b_{\pi}} \text{ is the sequence of the coefficients of $g_{P_{\pi(C)}}(q)$, $\pi \in \mathfrak{S}_r$} \},$$
    where $|C|=\sum_{i=1}^{r}c_i$. Then the poset $(\mathbf{B}(C),\preceq)$ is isomorphic to a refinement of the poset $(\mathfrak{S}_r,\leq)$, where $\leq$ is the weak order on $\mathfrak{S}_r$.
\end{thm}

We first prove the following lemma which will be used to prove \cref{thm:antichain_Bruhat_order}.
\begin{lem}\label{lem:antichain_Bruhat_order}
    Given an ordered sequence of $r$ distinct positive integers $C=(c_1,c_2,\dotsc,c_r)$. Let $\pi=(i,i+1) \in \mathfrak{S}_r$ be a transposition. Let $\mathbf{b} = (b_0,\dotsc,b_{c_1+\dotsc+c_r-1})$ and $\mathbf{b_{\pi}} = (b_0^{\prime},\dotsc,b_{c_1+\dotsc+c_r-1}^{\prime})$ be the coefficients of the cumulative generating functions $g_{P_C}(q)$ and $g_{P_{\pi(C)}}(q)$, respectively. If $c_i < c_{i+1}$, then $\mathbf{b_{\pi}} \preceq \mathbf{b}$ for $i=1,2,\dotsc,r-1$. 
\end{lem}
\begin{proof}
    For convenience, we write $\pi(C) = (d_1,d_2,\dotsc,d_r)$, where $d_i = c_{i+1}$, $d_{i+1} = c_i$ and $d_k = c_k$ for $k \neq i,i+1$. For $j = 1, 2, \ldots, r$, let $\mathcal{I}_j = \{s : \sum_{k=1}^{j-1}c_k \leq s < \sum_{k=1}^j c_k \}$ and $\mathcal{I^{\prime}}_j = \{s : \sum_{k=1}^{j-1}d_k \leq s < \sum_{k=1}^j d_k\}$. Since $c_k = d_k$ only differ for $k = i$ and $k= i+1$, \cref{prop:ordinalsumg} implies that $b_s = b_s^{\prime}$ for $s \in \mathcal{I}_{j}$ where $j \neq i, i+1$. 

    Notice that $\mathcal{I}_i \cup \mathcal{I}_{i+1} = \mathcal{I}^{\prime}_i \cup \mathcal{I}^{\prime}_{i+1}$ and $\mathcal{I}_i \subseteq \mathcal{I}^{\prime}_i$, it remains to check $b_s^{\prime} / b_s \leq 1$ holds for the following three cases: (1) $s \in \mathcal{I}_i$, (2) $s \in \mathcal{I^{\prime}}_i \setminus \mathcal{I}_i$, and (3) $s \in \mathcal{I^{\prime}}_{i+1}$. This will imply that $\mathbf{b_{\pi}} \preceq \mathbf{b}$. For the last case, we obtain the equality $b_s = b_s^{\prime}$ by \cref{prop:ordinalsumg} immediately. The calculation for the first two cases is presented below. 

    Let $x^{\overline{n}} = \prod_{k=1}^{n}(x+k-1)$ denote the rising factorial of $x$.
    
    \textbf{Case 1:} $s \in \mathcal{I}_i$. We may write $s = c_1 +\dotsc + c_{i-1} + t$, where $0 \leq t \leq c_i-1$. Then for each such $t$,
        \begin{align*}
            \frac{b_{s}^{\prime}}{b_s} & = \frac{(d_1+\dotsc+d_i)! \prod_{m=i+1}^{r} \frac{(d_m+s)!}{s!}}{(c_1+\dotsc+c_i)! \prod_{m=i+1}^{r} \frac{(c_m+s)!}{s!}}\\
            &= \frac{(c_1 + \dotsc + c_{i-1} + c_{i+1})! (c_{i}+s)!}{(c_1 + \dotsc + c_{i-1} + c_{i})! (c_{i+1}+s)!}\\
            & = \frac{(c_1 + \dotsc + c_{i-1} + c_{i+1})! (c_1 +\dotsc + c_{i-1} + c_{i} + t)!}{(c_1 + \dotsc + c_{i-1} + c_{i})! (c_1 +\dotsc + c_{i-1} + c_{i+1} + t)!}\\
            &= \frac{(c_1 + \dotsc + c_{i-1} + c_{i}+1)^{\overline{t}}}{(c_1 + \dotsc + c_{i-1} + c_{i+1}+1)^{\overline{t}}} \leq 1,
        \end{align*}
        because $c_i < c_{i+1}$. 

    \textbf{Case 2:} $s \in \mathcal{I^{\prime}}_{i} \setminus \mathcal{I}_{i}$. We may write $s = c_1 +\dotsc + c_{i} + t$, where $0 \leq t \leq c_{i+1}-c_{i}-1$. Then for each such $t$,
            \begin{align*}
                \frac{b_{s}^{\prime}}{b_s} & = \frac{(d_1+\dotsc+d_{i})! \prod_{m=i+1}^{r} \frac{(d_m+s)!}{s!}}{(c_1+\dotsc+c_{i+1})! \prod_{m=i+2}^{r} \frac{(c_m+s)!}{s!}} \\
                &=
                \frac{(c_1 + \dotsc + c_{i-1} + c_{i+1})! (c_{i}+s)!}{(c_1 + \dotsc +c_{i+1})! s!}\\
                & = \frac{(c_1 + \dotsc + c_{i-1} + c_{i+1})! (c_1 +\dotsc + c_{i} + c_{i} + t)!}{(c_1 + \dotsc + c_i + c_{i+1})! (c_1 +\dotsc + c_{i} + t)!}\\
                &= \frac{(c_1 + \dotsc + c_{i-1} + c_{i}+t+1)^{\overline{c_{i+1}-c_{i}-t}}}{(c_1 + \dotsc + c_{i} + c_{i}+t+1)^{\overline{c_{i+1}-c_{i}-t}}} \leq 1,
            \end{align*}
            by the same reasoning in Case 1.
            This completes the proof of \cref{lem:antichain_Bruhat_order}.
\end{proof}
\begin{proof}[Proof of \cref{thm:antichain_Bruhat_order}]
    Without loss of generality, we assume the elements of $C$ are written in the increasing order, $c_1 < c_2 < \cdots < c_r$. The permutations $\pi \in \mathfrak{S}_r$ in this proof will be written in the one-line notation $\pi = p_1p_2 \cdots p_r$. 
    
    Define the map $\phi: (\mathfrak{S}_r,\leq) \rightarrow (\mathbf{B}(C),\preceq)$ by sending a permutation $\pi = p_1p_2 \dotsc p_r$ to $\mathbf{b}_{\mathsf{rev}(\pi)}$, where $\mathsf{rev}(\pi) = p_rp_{r-1} \dotsc p_1$ is the reverse of $\pi$, and $\mathbf{b}_{\mathsf{rev}(\pi)}$ is the sequence of the coefficients of $g_{ P_{\mathsf{rev}(\pi)(C)}}(q)$. Let $\sigma$ be the adjacent transposition that swapped the elements at positions $i$ and $i+1$. Let $\pi \in \mathfrak{S}_r$ be a permutation such that $\pi \leq \sigma\pi$ in the weak order. One may write $\pi = p_1p_2 \dotsc p_r$ with $p_i < p_{i+1}$, and $\sigma\pi = p_1 \dotsc p_{i-1}p_{i+1}p_{i}p_{i+2} \dotsc p_r$.
    
    We show below that if $\pi \leq \sigma\pi$ in $(\mathfrak{S}_r,\leq)$, then $\phi(\pi) \preceq \phi(\sigma\pi)$ in $(\mathbf{B}(C),\preceq)$. Intuitively, $\mathsf{rev}(\pi)(C) = \{c_{p_r},\dotsc,c_{p_1}\}$ and $\mathsf{rev}(\sigma\pi)(C) = \{c_{p_r},\dotsc,c_{p_{i+2}},c_{p_{i}},c_{p_{i+1}},c_{p_{i-1}},\dotsc,c_{p_1}\}$. Since $p_i < p_{i+1}$ and $c_{p_i} < c_{p_{i+1}}$ (by the assumption that $c_i$'s are increasing as $i$ increases), by \cref{lem:antichain_Bruhat_order}, we obtain $\mathbf{b}_{\mathsf{rev}(\pi)} \preceq \mathbf{b}_{\mathsf{rev}(\sigma\pi)}$.
    
    Therefore, $\phi(\pi) = \mathbf{b}_{\mathsf{rev}(\pi)} \preceq \mathbf{b}_{\mathsf{rev}(\sigma\pi)} = \phi(\sigma\pi)$. The poset $(\mathbf{B}(C),\preceq)$ is thus isomorphic to a refinement of $(\mathfrak{S}_r,\leq)$.
\end{proof}

We would like to point out that $(\mathbf{B}(C),\preceq)$ is not a subposet of the strong order of $\mathfrak{S}_n$ in general. Take $C=(1,2,3,4)$ as an example (see the right picture of \cref{fig:Hassediagram} again); the cover relation $\mathbf{b}_{4123} \preceq \mathbf{b}_{3214}$, under the inverse of the map $\phi$ defined in the proof of \cref{thm:antichain_Bruhat_order}, does not relate in the strong order of $\mathfrak{S}_4$. One can also check that $(\mathbf{B}(C),\preceq)$ is not graded in general.

\section{Future Work}
\label{section:future-work}

We present some future directions from this work. In this paper, we propose the $(n-2)!$ conjecture (\cref{conj:(n-1)-refinement}), stating that the number of tangled $x$-labelings (the label of $x$ is fixed by $n-1$) of an $n$-element poset $P$ is bounded by $(n-2)!$. In \cref{section:inflated-rooted-forests} and \cref{section:shoelace-posets}, we prove that inflated rooted forest posets and shoelace posets satisfy the $(n-2)!$ conjecture. We also obtain the exact enumeration of tangled labelings of the $W$-poset (as a special case of the shoelace poset) in \cref{thm:w-tangled}. One can define inflated shoelace posets in analogy with inflated rooted forest posets. An interesting question would be to investigate whether inflated shoelace posets satisfy the $(n-2)!$ conjecture. Other general classes of posets that would be of interest to study include posets related to Young tableaux.

In \cref{section:generating-functions}, we explicitly determine the number of $k$-sorted labelings of the poset $T_s \oplus P$ from $P$ (attach $s$ minimal elements to $P$) via the matrix multiplication stated in \cref{thm:attach_k_minimal_gf}. However, obtaining the number of $k$-sorted labelings of the poset $P \oplus T_s$ from $P$ (attach $s$ maximal elements to $P$) does not seem to have such a nice pattern. There may exist some other ways to express them. We leave this direction to be pursued by the interested reader.

In \cref{sec:ordinalsum}, we introduce the new poset structure $(\mathbf{B}(C),\preceq)$ and show that it contains a subposet which is isomorphic to the weak order of the symmetric group (\cref{thm:antichain_Bruhat_order}). It would be an interesting follow-up to fully characterize our poset $(\mathbf{B}(C),\preceq)$ as a poset on permutations.

\section*{Acknowledgments}
This work was initiated at the 2023 Graduate Research Workshop in Combinatorics, which was supported in part by NSF grant \# 1953445, NSA grant \# H98230-23-1-0028, and the Combinatorics Foundation. The last author was supported in part by the Natural Sciences and Engineering Research Council of Canada.

We would like to thank the workshop organizers and the University of Wyoming for their hospitality during our stay in Laramie. We also thank Joel Jeffries, Lauren Kimpel, and Nick Veldt for their contributions to the beginning of this project.


\end{document}